%% file: main_archiv.tex
\documentclass[11pt,final]{amsart}

\usepackage{graphicx}
\usepackage{amssymb,amsmath}
\usepackage{epstopdf} 
\usepackage{mathptmx}      
\usepackage[top=26mm, bottom=26mm, left=23mm, right=25mm]{geometry}

\usepackage{multirow}
\usepackage{graphicx}
\usepackage[dvips]{color}
\usepackage{bm}

\newtheorem{thm}{Theorem}[section]

\theoremstyle{definition}
\newtheorem{remark}{Remark}
\numberwithin{remark}{section}
\newtheorem{definition}{Definition}
\numberwithin{definition}{section}
\newtheorem{pro}{Proposition}
\numberwithin{pro}{section}

\input{macros.tex}         

\begin{document}

\begin{center}
\textbf{\large A spectrally accurate direct solution technique for
frequency-domain scattering problems with variable media}

\lsp

{\small A. Gillman$^1$,  A. H. Barnett$^1$, and P. G. Martinsson$^2$\\
$^1$ Department of Mathematics, Dartmouth College \quad $^2$ Department of Applied Mathematics, University of Colorado at Boulder}

\lsp

\begin{minipage}{135mm}
\noindent\textbf{Abstract:}
This paper presents a direct solution technique for the scattering of time-harmonic
waves from a bounded region of the plane in which the wavenumber varies
smoothly in space.
%
The method constructs the interior Dirichlet-to-Neumann (DtN) map
for the bounded region via
bottom-up recursive merges of (discretization of) certain boundary operators on a quadtree of boxes.
These operators take the form of impedance-to-impedance (ItI) maps.
Since ItI maps are unitary, this formulation is inherently numerically
stable, and is immune to problems of artificial internal resonances.
The ItI maps on the smallest (leaf) boxes are built by spectral collocation
on tensor-product grids of Chebyshev nodes.
%
%
At the top level the DtN map is recovered from the ItI map and coupled
to a boundary integral formulation of the free space exterior problem,
to give a provably second kind equation.
Numerical results indicate that the scheme can solve challenging problems
$70$ wavelengths on a side to 9-digit accuracy
with 4 million unknowns, in under 5 minutes on a desktop workstation.
Each additional solve corresponding to a different incident wave
(right-hand side) then requires only
$0.04$ seconds. 
\end{minipage}
\end{center}

\section{Introduction}
\label{sec:intro}
\subsection{Problem formulation}
\label{sec:formulation}
Consider time-harmonic waves propagating in a medium where the wave speed varies smoothly, but
is constant outside of a bounded domain $\Omega\subset\RR$. This manuscript presents a technique
for numerically solving the scattering problem in such a medium.  Specifically, we
seek to compute the scattered wave $\us$ that results when a given incident wave $\ui$
(which satisfies the free space Helmholtz equation) impinges upon the region with variable wave speed,
as in Figure~\ref{fig:scattering_geometry}.
Mathematically, the scattered field $\us$ satisfies the variable coefficient Helmholtz equation
\begin{equation}
\label{eq:x1}
\Delta \us(\xx) + \kappa^{2}(1 - b(\xx))\us(\xx) = \kappa^{2}b(\xx)\ui(\xx),\qquad\xx \in \mathbb{R}^{2},
\end{equation}
and the outgoing Sommerfeld radiation condition
\begin{equation}
\label{eq:x2}
\frac{\partial \us}{\partial r} - i \kappa \us = o(r^{-1/2}), \qquad r:=|\xx| \to \infty,
\end{equation}
uniformly in angle. The real number $\kappa$ in (\ref{eq:x1}) and (\ref{eq:x2}) is the free space
wavenumber (or frequency),
and the so called ``scattering potential'' $b = b(\xx)$ is a given smooth function that specifies
how the wave speed (phase velocity)
$v(\xx)$ at the point $\xx \in \mathbb{R}^{2}$ deviates from the free space
wave speed $v_{\rm free}$,
\begin{equation}
\label{eq:def_b}
b(\xx) = 1 - \left(\frac{v_{\rm free}}{v(\xx)}\right)^{2}.
\end{equation}
One may interpret $\sqrt{1-b}$ as a spatially-varying refractive index.
Observe that $b$ is identically zero outside $\Omega$.  Together, equations (\ref{eq:x1}) and (\ref{eq:x2})
completely specify the problem.
When $1-b(\xx)$ is real and positive for all $\xx$,
the problem is known to have a unique solution for each positive $\kappa$
\cite[Thm.~8.7]{coltonkress}.


The transmission problem \eqref{eq:x1}-\eqref{eq:x2}, and its generalizations,
have applications in acoustics, electromagnetics, optics, and quantum mechanics.
Some specific applications include underwater acoustics \cite{underwater},
ultrasound and microwave tomography \cite{3dmicrowave,wadbro10},
wave propagation in metamaterials and photonic crystals, and seismology \cite{seismic3D}.
The solution technique in this paper is high-order accurate, robust and computationally highly efficient.
It is based on a direct (as opposed to iterative) solver, and thus is particularly effective
when the response of a given potential $b$ to multiple incident waves $\ui$ is desired,
as arises in optical device characterization, or computing radar scattering
cross-sections.
The complexity of the method is
$O(N^{3/2})$ where $N$ is the number of discretization points in $\Omega$.
Additional solves with the same scattering potential $b$ and wavenumber $\kappa$
require only $O(N)$
operations. (Further reductions in asymptotic complexity can sometimes be attained; see section \ref{sec:N}.)
For simplicity of presentation, the solution technique is presented in $\mathbb{R}^2$; however, the method
can be directly extended to $\mathbb{R}^3$.

\begin{remark}
Equation (\ref{eq:x1}) is derived by requiring that the total field $u = \us + \ui$ satisfy
the variable coefficient Helmholtz equation
\begin{equation}
\label{eq:totalfield}
\Delta u(\xx) + \kappa^{2}\,\left(\frac{v_{\rm free}}{v(\xx)}\right)^{2}u(\xx) = 0,\qquad \xx \in \mathbb{R}^{2}.
\end{equation}
Plugging the condition that the incident field $\ui$ satisfies the free space equation
 $(\Delta + \kappa^{2})\ui = 0$ inside $\Omega$, and the definition
of the scattering potential (\ref{eq:def_b}) into (\ref{eq:totalfield}) results in (\ref{eq:x1}).
\end{remark}

\begin{figure}
\setlength{\unitlength}{1mm}
\begin{picture}(140,50)
\put(10,00){\includegraphics[height=50mm]{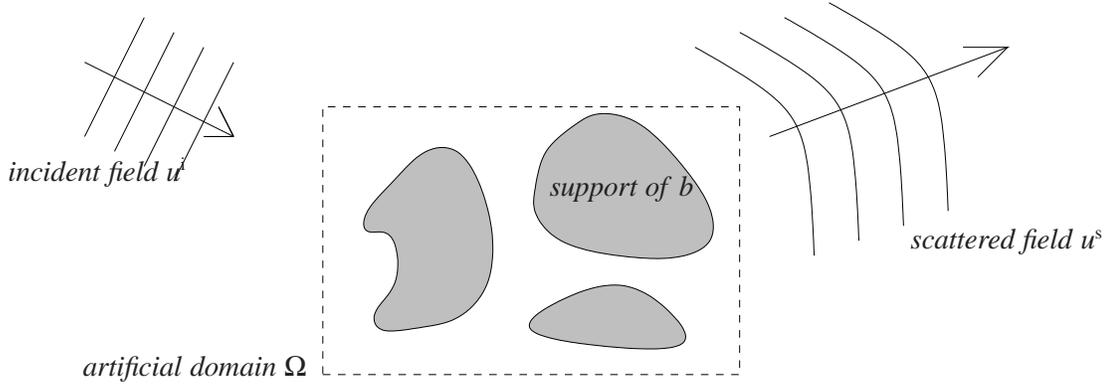}}
\put(00,26){\textit{incident field} $\ui$}
\put(120,17){\textit{scattered field} $\us$}
\put(72,24){\textit{support of} $b$}
\put(10,00){\textit{artificial domain} $\Omega$}
\end{picture}
\caption{Geometry of scattering problem. Waves propagate in a medium with constant
wave speed $v_{\rm free}$ everywhere except in the shaded areas, where the wave speed
is given by $v(\xx) = v_{\rm free}/\sqrt{1-b(\xx)}$ where $b = b(\xx)$ is a given, smooth, compactly
supported ``scattering potential.''
An incident field $\ui$ hits the scattering potential and induces the scattered field $\us$.
The dashed line marks the artificial domain $\Omega$ which encloses the support of the scattering potential.}
\label{fig:scattering_geometry}
\end{figure}

\subsection{Outline of proposed method}
We solve (\ref{eq:x1})-(\ref{eq:x2}) by splitting the problem into two parts,
namely a variable-coefficient problem on the bounded domain $\Omega$, and a constant coefficient
problem on the exterior domain $\Omega^{\rm c}:=\RR \backslash \overline{\Omega}$.
For each of the two domains, a solution operator in the form of a Dirichlet-to-Neumann
(DtN) map on the boundary $\pO$ is constructed.  These operators are then ``glued'' together
 on $\pO$ to form a solution operator for the full problem.
 The end result is a discretized boundary integral operator that takes as input the restriction of the
incoming field $\ui$ (and its normal derivative) to $\partial\Omega$, and
constructs the restriction of the scattered field $\us$ on $\partial\Omega$ (and its
normal derivative). Once these quantities are known, the total field can rapidly be
computed at any point $\xx \in \mathbb{R}^{2}$.

\subsubsection{Solution technique for the variable-coefficient problem on $\Omega$}
\label{sec:introint}
For the interior domain $\Omega$, we construct a solution operator for the following
homogeneous variable-coefficient boundary value problem where the unknown total field $u = \us + \ui$ satisfies
\begin{align}
\label{pdep}
\Delta u(\xx) +\kappa^2(1-b(\xx)) u(\xx) =&\ 0, & \xx &\in \Omega,
\\
\label{dir}
u(\xx) =&\ h(\xx) & \xx &\in \pO.
\end{align}
Note that, for now, we specify the Dirichlet data $h$;
when we consider the full problem $h$ will become an unknown that will also
be solved for.
It is known that, for all but a countable set of wavenumbers, the interior Dirichlet BVP \eqref{pdep}-\eqref{dir} has a unique solution $u$
\cite[Thm.~4.10]{mclean2000}.
The values $\{\kappa_j\}_{j=1}^{\infty}$ at which the solution is not unique are (the square roots of)
the interior Dirichlet eigenvalues of the penetrable domain $\Omega$;
we will call them
{\em resonant wavenumbers}.

\begin{definition}[interior Dirichlet-to-Neumann map]
Suppose that $\kappa>0$ is not a resonant wavenumber
of $\Omega$. Then the interior
Dirichlet-to-Neumann (DtN) operator\footnote{this is also known as the Steklov--Poincar\'e operator \cite{mclean2000}.} $T_{\rm int}:H^1(\pO)\to L^2(\pO)$ is defined by
\be
T_{\rm int} h = u_n,
\label{T}
\ee
where $u$ the solution to \eqref{pdep}--\eqref{dir} corresponding to Dirichlet data $h$, for any $h\in H^1(\pO)$.
\label{d:dtn}
\end{definition}

\begin{remark}
The operator $T_{\rm int}$ is a pseudo-differential operator of order $+1$ \cite{friedlander}; that is,
in the limit of high-frequency boundary data it behaves like a differentiation operator on $\pO$.
The boundedness as a map $T_{\rm int}:H^1(\pO)\to L^2(\pO)$ holds
for $\Omega$ any bounded Lipschitz domain since the PDE is
strongly elliptic \cite[Thm.~4.25]{mclean2000}.
\label{r:pdo}
\end{remark}

In this paper, a discrete approximation to $T_{\rm int}$ is constructed via a variation of recent
composite spectral methods in \cite{ONspectralcomposite,2012_martinsson_spectralcomposite}
(which are also similar to \cite{2013_yu_chen_totalwave}).
These methods partition $\Omega$ into a collection of small ``leaf''
boxes and construct
approximate DtN operators for each box via a brute force calculation on a local
spectral grid.  The DtN operator for $\Omega$ is then
constructed via a hierarchical merge process.
Unfortunately, at any given $\kappa$,
each of the many leaves and merging subdomains may hit a resonance as described above, causing its
DtN to fail to exist.
As $\kappa$ approaches any such resonance the norm of the DtN grows without bound.
Thus a technique based on the DtN alone is not robust.
\begin{remark}
We remind the reader that any such ``box'' resonance is purely artificial and is caused by the introduction of the solution regions.
It is important to distinguish these from
resonances that the physical scattering problem \eqref{eq:x1}-\eqref{eq:x2}
itself might possess (e.g.\ due to nearly trapped rays),
whose effect of course cannot be avoided in {\em any}
accurate numerical method.
\end{remark}

One contribution of the present work is to present
a robust improvement to the methods of \cite{ONspectralcomposite,2012_martinsson_spectralcomposite},
built upon hierarchical merges of impedance-to-impedance (ItI) rather than
DtN operators; see section~\ref{sec:spec}.
The idea of using impedance coupling builds upon the work of
Kirsch--Monk \cite{kirschmonk94}.
ItI operators are inherently stable, with condition number $O(1)$,
and thus exclude the
possibility of inverting arbitrarily ill-conditioned matrices as in
the original DtN formulation.
For instance, in the \emph{lens} experiment of section \ref{sec:numerics},
the DtN method has condition numbers as large as $2 \times 10^{5}$, while for the new ItI
method the condition number never grows larger than $20$.
The DtN of the whole domain $\Omega$ is still needed; however, if $\Omega$ has a resonance,
the size of $\Omega$ can be changed slightly to avoid the resonance (see remark \ref{re:omegaRes}).

The discretization methods in \cite{ONspectralcomposite,2012_martinsson_spectralcomposite},
and in this paper are related to earlier work on spectral collocation methods on composite
(``multi-domain'') grids, such as, e.g., \cite{1998_Kopriva,2000_Yang}, and in particular
Pfeiffer et al \cite{2003_pfeiffer}. For a detailed review of the similarities and differences,
see \cite{2012_martinsson_spectralcomposite}.

\subsubsection{Solution technique for the constant coefficient problem on $\Omega^{\rm c}$}
\label{sec:introext}
On the exterior domain, we build a solution operator for the constant coefficient problem
\begin{align}
\label{eq:y1}
\Delta \us(\xx) + \kappa^{2}\us(\xx) =&\ 0, &&\xx \in \Omega^{\rm c},\\
\label{eq:y2}
\us(\xx) =&\ s(\xx), && \xx \in \partial \Omega,\\
\label{eq:y3}
\frac{\partial \us}{\partial r} - i \kappa \us =&\ o(r^{-1/2}), &&r:=|\mbf{x}| \to \infty,
\end{align}
obtained by restricting (\ref{eq:x1})-(\ref{eq:x2}) to $\Omega^{\rm c}$.
(Again, the Dirichlet data $s$ later will become an unknown that is solved
for.)
It is known that (\ref{eq:y1})-(\ref{eq:y3})
has a unique solution for every wavenumber $\kappa$ \cite[Ch.~3]{coltonkress}.
This means that the following DtN
for the exterior domain is always well-defined.

\begin{definition}[exterior Dirichlet-to-Neumann map]
Suppose that $\kappa>0$. The exterior DtN operator $T_{\rm ext}:H^1(\pO)\to L^2(\pO)$ is defined by
\be
T_{\rm ext} s= \us_n
\label{Te}
\ee
for $\us$ the unique solution to the exterior Dirichlet BVP (\ref{eq:y1})-(\ref{eq:y3}).
\end{definition}

Numerically, we construct an approximation to $T_{\rm ext}$ by reformulating (\ref{eq:y1})-(\ref{eq:y3})
as a boundary integral equation (BIE), as described in section \ref{sec:formulatText}, and then discretizing it using a Nystr\"om method based on a high order Gaussian composite quadrature \cite{gen_quad}.

\subsubsection{Combining the two solution operators}
Once the DtN operators $T_{\rm int}$ and $T_{\rm ext}$ have been determined (as described in
sections \ref{sec:introint} and \ref{sec:introext}), and the restriction of the incident field
to $\partial \Omega$ is given, it is possible to determine the scattered field on $\partial\Omega$
as follows. First observe that the total field $u = \us + \ui$ must satisfy
\be
T_{\rm int}(\ui|_\pO+\us|_\pO) = \ui_n+\us_n.
\label{eq:bdry1}
\ee
We also know that the scattered field $\us$ satisfies
\be
T_{\rm ext}\us|_\pO = \us_n.
\label{eq:bdry2}
\ee
Combining (\ref{eq:bdry1}) and (\ref{eq:bdry2}) to eliminate $\us_n$ results in the equation (analogous to \cite[Eq.~(2.12)]{kirschmonk94}),
\be
(T_{\rm int} - T_{\rm ext}) \us|_\pO = \ui_n - T_{\rm int}\ui|_\pO.
\label{eq:intro}
\ee

As discussed in Remark~\ref{r:pdo}, both $T_{\rm int}$ and $T_{\rm ext}$ have order $+1$.
Lamentably, this behavior {\em adds rather than cancels}
in \eqref{eq:intro}, so that
$(T_{\rm int} - T_{\rm ext})$ also has order $+1$, and is therefore unbounded.
This makes any numerical discretization of \eqref{eq:intro} 
ill-conditioned,
with condition number growing linearly in the number of boundary nodes.
To remedy this, we present in section~\ref{sec:2ndkind} a new method for combining
$T_{\rm int}$ and $T_{\rm ext}$ to give a provably second kind integral
equation, which thus gives a well-conditioned linear system.

Once the scattered field is known on the boundary,
the field at any exterior point may be found
via Green's representation formula; see section~\ref{sec:exterior}.
The interior transmitted wave $u$ may be reconstructed
anywhere in $\Omega$ by applying solution operators which were built as part of
the composite spectral method.

\subsection{Prior work}
Perhaps the most common technique for solving the scattering problem stated in
Section \ref{sec:formulation} is
to discretize the variable coefficient PDE (\ref{eq:x1})
via a finite element or finite difference method,
while approximating the radiation condition in one of many ways,
including
perfectly matched layers (PML) \cite{FDS_helmholtz_PML}, absorbing boundary
conditions (ABC) \cite{ABC_1977},
separation of variables or their perturbations \cite{nichollsnigam04},
local impedance conditions \cite{britt10},
or a Nystr\"om method \cite{kirschmonk94} (as in the present work).
However, the accuracy of finite element and finite difference
schemes for the Helmholtz equation is limited by so-called ``pollution''
(dispersion) error \cite{pollution,bayliss85},
demanding an increasing number of degrees of freedom per
wavelength in order to maintain fixed accuracy as wavenumber $\kappa$ grows.
In addition, while the resulting linear system is sparse, it is also large
and is often ill-conditioned in such a way that standard pre-conditioning techniques fail, although hybrid direct-iterative
solvers such as \cite{2011_engquist_ying_PML} have proven effective in certain environments.
While there do exist fast
direct solvers for such linear systems (for low wavenumbers $\kappa$)
\cite{2011_ying_nested_dissection_2D,2011_ying_nested_dissection_3D,2009_xia_superfast,2009_martinsson_FEM}, the accuracy of the
solution is limited by the discretization.
The performance of the solver worsens when increasing the order of the discretization---%
thus it is not feasible to use a high order discretization that would overcome the
above-mentioned pollution error.


Scattering problems on infinite domains are also commonly handled by
rewriting them as volume integral equations (e.g.~the
Lippmann--Schwinger equation) defined on a domain (such as $\Omega$)
that contains the support of the scattering potential \cite{Ang,2002_chen_direct_lippman_schwinger}.
This approach is appealing in that the Sommerfeld condition (\ref{eq:x2}) is enforced
analytically, and in that high-order discretizations can be implemented
without loss of stability \cite{2008_duan_rokhlin}. Principal drawbacks are
that the resulting linear systems have dense coefficient matrices, and tend
to be challenging to solve using iterative solvers \cite{2008_duan_rokhlin}.


\subsection{Outline}
Section \ref{sec:spec} describes in detail the stable hierarchical procedure for constructing an
approximation to the DtN map $T_{\rm int}$ for the interior problem (\ref{pdep})-(\ref{dir}).
Section \ref{sec:2ndkind} describes how boundary integral equation techniques are
used to approximate the DtN map $T_{\rm ext}$ for the exterior problem (\ref{eq:y1})-(\ref{eq:y3}),
how to couple the DtN maps $T_{\rm int}$ and $T_{\rm ext}$ to solve
the full problem (\ref{eq:x1})-(\ref{eq:x2}), and the proof (Theorem~\ref{t:2ndkind})
that the formulation is second kind.
Section \ref{sec:cost}
details the computational cost of the method and explains the reduced cost for multiple incident
waves.  Finally, section \ref{sec:numerics} illustrates
the performance of the method in several challenging scattering potential
configurations.

\section{Constructing and merging impedance-to-impedance maps}
\label{sec:spec}

This section describes a technique for building a discrete approximation to the Dirichlet-to-Neumann
(DtN) operator for the interior variable coefficient BVP (\ref{pdep})-(\ref{dir}) on a square
domain $\Omega$. It relies on the
hierarchical construction of impedance-to-impedance (ItI) maps; these are defined in section \ref{sec:iti}.
Section \ref{sec:part} defines a hierarchical tree on the domain $\Omega$.
Section \ref{sec:box} explains how the ItI maps are built on the (small) leaf boxes in the tree.
Section \ref{sec:merge} describes the merge procedure whereby the global ItI map is built,
and then how the global DtN map is recovered from the global ItI map.


\subsection{The impedance-to-impedance map}
\label{sec:iti}

We start by defining the ItI map on a general Lipschitz domain, and
giving some of its properties. (In this section only, $\Omega$ will
refer to such a general domain.)

\begin{pro}
Let $\Omega\subset\RR$ be a bounded Lipschitz domain, and $b(\xx)$ be real.
Let $\eta\in\C$, and $\re \eta\neq 0$. Then the interior Robin BVP,
\bea
[\Delta + \kappa^2(1-b(\xx))]u(\xx) &=& 0 \qquad \xx \in \Omega~,
\label{pder}
\\
u_n+i\eta u|_\pO&=& f \qquad \mbox{ on } \pO~,
\label{r}
\eea
has a unique solution $u$ for all real $\kappa>0$.
\label{p:robin}
\end{pro}
\begin{proof}
We first prove uniqueness. Consider $u$ a solution to the homogeneous problem
$f\equiv 0$. Then
using Green's 1st identity and \eqref{pder}, \eqref{r},
$$-i\eta\int_\pO |u|^2 = \int_\pO\overline{u}u_n =\int_\Omega
|\nabla u|^2 - \kappa^2(1-b)|u|^2 .
$$
Taking the imaginary part shows that $u$, and hence $u_n$, vanishes on
$\pO$, hence $u\equiv 0$ in $\Omega$ by unique continuation of the Cauchy data.
Existence of $u\in H^1(\Omega)$ now follows
for data $f\in H^{-1/2}(\pO)$ from the Fredholm alternative,
as explained in this context by McLean \cite[Thm.~4.11]{mclean2000}.
\end{proof}

\begin{definition}[interior impedance-to-impedance map]
Fix $\eta\in\C$, and $\re \eta\neq 0$. Let
\bea
f&:=& u_n+i\eta u|_\pO
\label{f}
\\
g&:=&u_n-i\eta u|_\pO
\label{g}
\eea
be Robin traces of $u$. We refer to $f$ and $g$ as the ``incoming'' and ``outgoing'' (respectively)
impedance data.
For any $\kappa>0$, the \emph{interior ItI operator} $R:L^2(\pO)\to L^2(\pO)$ is defined by
\be
R f = g
\label{R}
\ee
for $f$ and $g$ the Robin traces of $u$ the solution of \eqref{pder}--\eqref{r},
for all $f \in L^2(\pO)$.
\end{definition}
We choose the impedance parameter $\eta$ (on dimensional grounds)
to be $\eta = \kappa$.  Numerically, in what follows,
we observe very little sensitivity to the exact value or sign of $\eta$.

For the following, we need the result that
the DtN map $T_{\rm int}$ is self-adjoint for real $\kappa$ and $b(\xx)$.
This holds since, for any functions $u$ and $v$ satisfying
$(\Delta + \kappa^2(1-b))u = 0$ in $\Omega$ and $(\Delta + \kappa(1-b))v = 0$
in $\Omega$,
\begin{multline*}
0= \int_\Omega \overline{v} (\Delta+\om^2(1-b))u - u(\Delta+\overline{\kappa^2(1-b)})\overline{v} = \int_\pO \overline{v}u_n - u\overline{v_n}\\
= (v|_\pO,\dtn_{\rm int} u|_\pO) - (\dtn_{\rm int} v|_\pO, u|_\pO)
\end{multline*}
by Green's second identity.
This allows us to prove the following property of the ItI map that will be
the key to the numerical stability of the method.

\begin{pro} Let $\Omega$ be a bounded Lipschitz domain, let $b(\xx)$ be real,
and let $\eta\in\C$ and $\re \eta\neq 0$.
Then the ItI map $R$ for $\Omega$ exists for all real frequencies $\kappa$,
and is unitary whenever $\eta$ is also real.
\end{pro}
\begin{proof}
Existence of $\iti$ for all real $\kappa$ follows from Proposition~\ref{p:robin}.
To prove $\iti$ is unitary, we insert the definitions of $f$ and $g$ into
\eqref{R} and use the definition of the DtN to rewrite $u_n = T_{\rm int} \ub$,
giving
\be
\iti(T_{\rm int} + i\eta)\ub = (T_{\rm int} - i\eta)\ub~,
\label{eq:ItI}
\ee
which holds for any data $\ub\in H^1(\pO)$.
Thus the ItI map is given in operator form by
\be
\iti = (T_{\rm int} - i\eta)(T_{\rm int} + i\eta)^{-1}
~.
\label{itidtn}
\ee
Since $T_{\rm int}$ is self-adjoint and $\eta$ is real, this formula shows
that $\iti$ is unitary.
\end{proof}

As a unitary operator, $R$ has unit operator $L^2$-norm, pseudo-differential
order $0$, and eigenvalues lying on the unit circle. From \eqref{itidtn}
and the pseudo-differential order of $T_{\rm int}$
one may see that the eigenvalues of $R$ accumulate only at $+1$.

\begin{figure} 
\includegraphics[width=\textwidth]{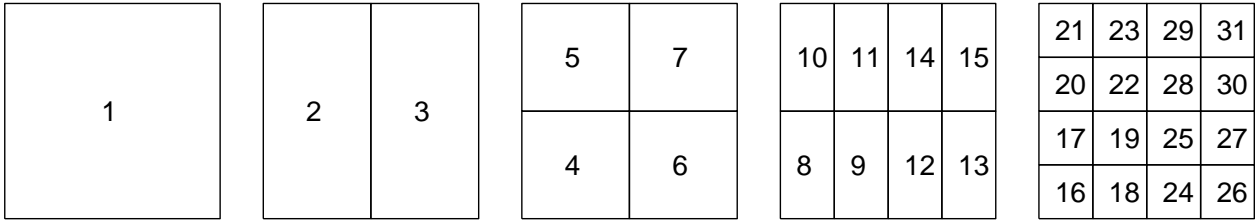}
\caption{The square domain $\Omega$ shown split hierarchically into boxes.
The case where there are $4 \times 4$ leaf boxes is shown,
ie there are $M=2$ levels.
These are then gathered into a binary tree of successively larger boxes
as described in Section \ref{sec:spec}. One possible enumeration
of the boxes in the tree is shown, but note that the only restriction is
that if box $\tau$ is the parent of box $\sigma$, then $\tau < \sigma$.}
\label{fig:tree_numbering}
\end{figure}

\begin{figure} 
\includegraphics[width = \textwidth]{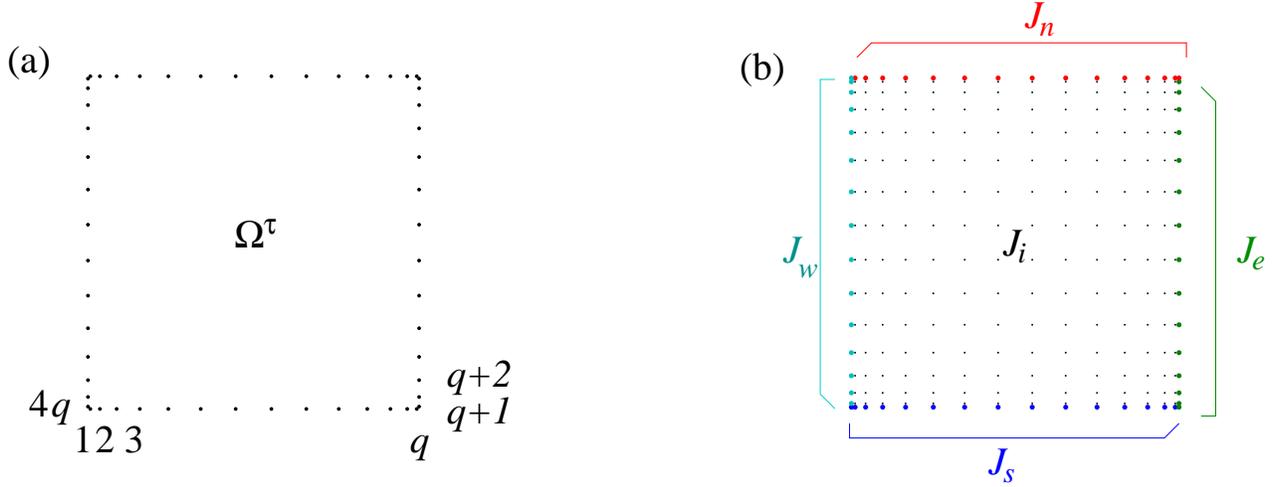}
\caption{\label{f:leafgeom} (a) The set of $4\Ng$ Gauss--Legendre points used
to represent the ItI on the boundary of a leaf box $\Omega^\tau$.
The case $\Ng = 14$ is shown.
(b) Chebyshev discretization used for the PDE spectral solution on the same
leaf box, for the case $\Nc = 16$.
Interior nodes $J_i$ are shown by small black dots.
The four sets of $\Nc-1$ Chebyshev
boundary nodes are shown (each a different color);
each set includes the start corner
(in counter-clockwise ordering) but not the end corner.
The set of all Chebyshev boundary nodes is $J_b = [J_s, \;J_e,\; J_n,\; J_w]$.
}
\end{figure}

\subsection{Partitioning of $\Omega$ into hierarchical tree of boxes}
\label{sec:part}

Recall that $\Omega$ is the square domain containing the support of $b$.
We partition $\Omega$ into a collection of $4^M$ equally-sized
square boxes called \textit{leaf boxes}, where $M$ sets the number of levels;
see Figure~\ref{fig:tree_numbering}.
We place $\Ng$ Gauss--Legendre interpolation nodes on each edge of each leaf,
which will serve to discretize all interactions of this leaf with its neighbors;
see Figure~\ref{f:leafgeom}(a).
The size of the leaf boxes, and the parameter $\Ng$,
should be chosen so that any potential transmitted wave $u$, as well as its first derivatives,
can be accurately interpolated on each box edge from their values at these nodes. 


Next we construct a binary tree over the collection of leaf boxes. This is achieved
by forming the union of adjacent pairs boxes (forming rectangular boxes), then
forming the pairwise union of the rectangular boxes.  The result is a collection
of squares with twice the side length of a leaf box.  The process is continued
until the only box is $\Omega$ itself, as in Figure
\ref{fig:tree_numbering}.  The boxes should be ordered so
that if $\tau$ is a parent of a box $\sigma$, then $\tau < \sigma$. We
also assume that the root of the tree (i.e.~the full box $\Omega$) has
index $\tau=1$. Let $\Omega^{\tau}$ denote the domain associated with box $\tau$.

\begin{remark}
The method easily generalizes to rectangular boxes, and to more complicated domains $\Omega$
in the same manner as \cite{2012_martinsson_spectralcomposite}.
\end{remark}

\subsection{Spectral approximation of the ItI map on a leaf box}
\label{sec:box}

Let $\Omega^\tau$ denote a single leaf box,
and let $\sss= \sss^\tau$ and $\ww=\ww^\tau$ be a pair of
vectors of associated incoming and outgoing impedance data, sampled at the
$4q$ Gauss--Legendre boundary nodes, with entries ordered in a
counter-clockwise fashion starting from the leftmost node of the
bottom edge of the box, as in Figure~\ref{f:leafgeom}(a).
In this section, we describe a technique for constructing
 a matrix approximation to the ItI operator on this leaf box.
Namely, we build a
$4\,\ngauss \times4\, \ngauss$ matrix $\mtx{R}$  such that
$$\ww \approx \mtx{R} \sss$$ holds to high-order accuracy,
for all incoming data vectors $\sss\in\R^{4\Ng}$ corresponding to smooth transmitted wave
solutions $u$.


First, we discretize the PDE \eqref{pder} on the square leaf box $\Omega^\tau$
using a  spectral method on a
$\Nc\times\Nc$ tensor product Chebyshev grid filling the box,
as in Figure~\ref{f:leafgeom}(b),
comprised of the nodes at locations
$$\left(\frac{a+b}{2}+hx_i, \;\frac{c+d}{2}+hx_j \right), \qquad i,j = 1,\ldots,\Nc$$
where $x_j := \cos\left(\frac{\pi (j-1)}{\Nc-1}\right)$, $j = 1,\ldots,\Nc$ are the
Chebyshev points on $[-1,1]$.
We label the Chebyshev node locations $\xx_j\in\RR$, for $j=1,\ldots,\Nc^2$.
For notational purposes, we order these nodes in the following fashion:
the indices $J_b = \{1,2,\ldots, 4(\Nc-1)\}$ correspond to the Chebyshev
nodes lying on the
boundary of $\Omega^\tau$, ordered counter-clockwise starting from
the node located at the south-west corner $(a,c)$.
The remaining $(\Nc-2)^2$ interior nodes have indices $J_i = \{4(\Nc-1)+1, \ldots, \Nc^2\}$ and
may be ordered arbitrarily (a Cartesian ordering is convenient).

Let $\Dx,\Dy \in \mathbb{R}^{\Nc^2\times \Nc^2}$ be the standard
spectral differentiation
matrices constructed on the full set of Chebyshev
nodes, which approximate the $\partial/\partial x_1$
(horizontal) and $\partial/\partial x_2$ (vertical) derivative operators, respectively.
As explained in \cite[Ch.~7]{tref},
these are constructed from Kronecker products of the $\Nc\times \Nc$ identity
matrix and $h^{-1} \mtx{D}$,
where $\mtx{D}\in\mathbb{R}^{\Nc\times \Nc}$
is the usual one-dimensional differentiation matrix
on the nodes $\{x_i\}_{i=1}^\Nc$.
Namely $\mtx{D}$ has entries $\mtx{D}_{ij} = \frac{w_j}{w_i(x_i-x_j)}$
where $\{w_j\}_{j=1}^\Nc = [1/2,-1,1,-1,\ldots,(-1)^\Nc,(-1)^{\Nc-1}/2]$ is
the vector of barycentric weights for the Chebyshev nodes (see \cite[Ch.~6]{tref} and \cite[Eqn.(8)]{salzer}.)
Let the matrix $\mtx{A}\in\mathbb{R}^{\Nc^2\times \Nc^2}$ be the spectral discretization
of the operator $\Delta + \kappa^2(1-b(\mbf{x}))$ on the product Chebyshev grid,
namely $$\mtx{A} = (\Dx)^2 + (\Dy)^2 + \mbox{ diag }\{\kappa^2(1-b(\mbf{x}_j))\}_{j=1}^{\Nc^2}~,$$
where ``diag S'' indicates the diagonal matrix whose entries are the elements of the ordered set $S$.
\begin{remark}
The matrices $\Dx$, $\Dy$, and $\mtx{A}$ must have rows and columns
ordered as explained above (i.e.\ boundary then interior) for
the Chebyshev nodes;
this requires permuting rows and columns of the matrices constructed
by Kronecker products.
For example, the structure of $\mtx{A}$ is
$$\left[
\begin{array}{ll}
\mtx{A}_{bb} & \mtx{A}_{bi} \\
\mtx{A}_{ib} & \mtx{A}_{ii}
\end{array}
\right]
~,$$
where the notation $\mtx{A}_{bi}$ indicates the submatrix block $\mtx{A}(J_b,J_i)$, etc.
\end{remark}

We now break the $4(\Nc-1)$
boundary Chebyshev nodes into four sets $J_b = [J_s, \;J_e,\; J_n,\; J_w]$,
denoting the south, east, north, and west edges,
as in Figure~\ref{f:leafgeom}(b).
Note that $J_s$ includes the south-western corner $J_s(1)$ but not the
south-eastern corner (which in turn is the first element of $J_e$), etc.

We are now ready to derive the linear system required for constructing
the approximate ItI operator. 
We first build a matrix $\NN\in\R^{4(\Nc-1)\times \Nc^2}$ which maps values of
$u$ at all Chebyshev nodes to the outgoing normal derivatives at the boundary
Chebyshev nodes, as follows,
\be
\NN = \left[\begin{array}{r}-\Dy(J_s,:)\\\Dx(J_e,:)\\\Dy(J_n,:)\\-\Dx(J_w,:)\end{array}\right],
\label{N}
\ee
where (as is standard in MATLAB) the notation $\mtx{A}(S,:)$
denotes the matrix formed from the subset of rows of a matrix $\mtx{A}$
given by the index set $S$.
Then, recalling \eqref{r},
the matrix $\mtx{F}\in\R^{4(\Nc-1)\times \Nc^2}$ which maps
the values of
$u$ at all Chebyshev nodes to incoming impedance data
on the boundary Chebyshev nodes is
\be
\mtx{F} \;=\; \NN \;+\; i\eta \mtx{I}_{p^2}(J_b,:)~,
\ee
where $\mtx{I}_{\Nc^2}$ denotes the identity matrix of size $\Nc^2$.
Using $\uu\in\R^{\Nc^2}$ for the vector of $u$ values at all Chebyshev nodes,
the linear system for the unknown $\uu$ that imposes the spectral discretization of the PDE at all interior nodes,
and incoming impedance data $\sss_c\in\R^{4(\Nc-1)}$
at the boundary Chebyshev nodes,
is
\be
\mtx{B} \uu = \left[\begin{array}{l}\sss_c \\ \mtx{0}\end{array}\right]
\label{eq:localsolve}
\ee
where $\mtx{0}$ is an appropriate column vector of zeros, and
the square size-$\Nc^2$ system matrix is
$$
\mtx{B} := \left[\begin{array}{l}
\mtx{F} \\ 
\mtx{A}(J_i,:)
\end{array}\right].
$$
\begin{remark}
At each of the four corner nodes, only one boundary condition is imposed,
namely the one associated with the edge lying in the counter-clockwise
direction. This results in a square linear system, which we observe
is around twice
as fast to solve as a similar-sized rectangular one.
\end{remark}
To construct the $\Nc^{2} \times 4(\Nc-1)$ ``solution matrix''
$\mtx{X}$ for the linear system, we
solve \eqref{eq:localsolve} for each unit vector in $\R^{4(\Nc-1)}$,
namely
$$\mtx{B}\mtx{X} = \left[\begin{array}{l}
\mtx{I}_{4\Nc-4} \\ 
\mtx{0}_{(\Nc-2)^2\times(4\Nc-4)} \end{array}\right]
~.
$$
In practice, $\mtx{X}$ is found using the backwards-stable solver available via MATLAB's {\tt mldivide}
command.  If desired, the tabulated solution $\uu$ can now
be found at all the Chebyshev nodes
by applying $\mtx{X}$ to the right hand side of \eqref{eq:localsolve}.

Recall that the goal is to construct matrices that act on boundary data on
Gauss (as opposed to Chebyshev) nodes.
With this in mind, let $\mtx{P}$
be the matrix which performs Lagrange polynomial
interpolation \cite[Ch.~12]{tyrt}
from $\Ng$ Gauss to $\Nc$ Chebyshev points on a single edge,
and let $\mtx{Q}$ be the matrix from Chebyshev to Gauss points.
Let $\mtx{P}_0\in\mathbb{R}^{(\Nc-1)\times \Ng}$ be $\mtx{P}$ with the last row
omitted. For example, $\mtx{P}_0$ maps from Gauss points on the south edge to
the Chebyshev points $J_s$. 

Then the solution matrix which takes incoming impedance data on Gauss nodes to
the values $\uu$ at all Chebyshev nodes is
\be
\mtx{Y} = \mtx{X}
\left[\begin{array}{llll}\mtx{P}_0& & &\\&\mtx{P}_0&&\\&&\mtx{P}_0&\\&&&\mtx{P}_0\end{array}\right]
~.
\label{Y}
\ee
Finally, we must extract outgoing impedance data on Gauss nodes
from the vector $\uu$,
to construct an approximation $\itim$ to the full ItI map on the Gauss nodes.
This is done by extracting (as in \eqref{N}) the relevant rows of the spectral differentiation matrices,
then interpolating back to Gauss points.
Let $J_s' := [J_s,J_e(1)]$ be the indices of all $\Nc$ Chebyshev nodes on the
south edge, and correspondingly for the other three edges.
Then the index set $J_b' := [J_s',\,J_e',\,J_n',\,J_w']$
counts each corner {\em twice}.\footnote{Including both endpoints allows more accurate
interpolation back to Gauss nodes;
functions on each edge are available at all Chebyshev nodes for that edge.}
Then let $\mtx{G}\in \mathbb{R}^{4\Nc\times \Nc^2}$ be the matrix mapping values of $\uu$ to the outgoing
impedance data with respect to each edge, given by
$$
\mtx{G} =  \left[\begin{array}{r}-\Dy(J_s',:)\\\Dx(J_e',:)\\\Dy(J_n',:)\\-\Dx(J_w',:)\end{array}\right]
- i\eta \mtx{I}_{\Nc^2}(J_b',:)
~.
$$
Then, in terms of \eqref{Y},
$$
\itim =
\left[\begin{array}{llll}\mtx{Q}& & &\\&\mtx{Q}&&\\&&\mtx{Q}&\\&&&\mtx{Q}\end{array}\right]
\mtx{G} \mtx{Y}
$$
is the desired spectral approximation to the ItI map on the leaf box.

The computation time is dominated by the solution step for $\mtx{X}$, which takes
effort $O(\Nc^6)$.
We observe empirically that one must choose $\Nc > \Ng+1$ in order that $\itim$ not acquire a spurious
numerical null space. We typically choose $\Ng = 14$ and $\Nc = 16$.

\subsection{Merging ItI maps}
\label{sec:merge}
Once the approximate ItI maps are constructed
on the boundary Gauss nodes on the leaf boxes, the ItI map defined on $\Omega$
is constructed by merging two boxes at a time,
moving up the binary tree as described in section
\ref{sec:part}.
This section demonstrates the purely local construction
of an ItI operator for a box from the ItI operators of
its children.


We begin by introducing some notation.
Let $\Omega^\tau$  denote a box with children $\Omega^\alpha$ and
$\Omega^\beta$ where $\Omega^\tau = \Omega^\alpha \cup \Omega^\beta$.
  For concreteness, consider the case where
$\Omega^{\alpha}$ and $\Omega^\beta$ share
a vertical edge.
As shown in Figure \ref{fig:siblings_notation},
the Gauss points on $\partial\Omega^{\alpha}$ and $\partial\Omega^{\beta}$ are partitioned into three sets:
\begin{tabbing}
\mbox{}\hspace{5mm}\= $J_{1}$:\hspace{4mm} \=
Boundary nodes of $\Omega^{\alpha}$ that are not boundary nodes of $\Omega^{\beta}$.\\
\> $J_{2}$: \> Boundary nodes of $\Omega^{\beta}$ that are not boundary nodes of $\Omega^{\alpha}$.\\
\> $J_{3}$: \> Boundary nodes of both $\Omega^{\alpha}$ and $\Omega^{\beta}$ that are \textit{not} boundary nodes of the union box
$\Omega^{\tau}$.
\end{tabbing}
Define interior and exterior outgoing data via
$$
\ww^i = \ww_3^\alpha \qquad \mbox{ and } \qquad \ww^{e} = \vtwo{\ww^\alpha_1}{\ww^\beta_2}.
$$
The incoming vectors $\sss^i$ and $\sss^e$ are defined similarly.  The goal is to obtain an
equation mapping $\sss^e$ to $\ww^e$.
Since the ItI operators for $\Omega^\alpha$ and $\Omega^\beta$ have previously been constructed, we have the following two equations
\be
 \mtwo{\mtx{R}^\alpha_{11}}{\mtx{R}^\alpha_{13}}{\mtx{R}^\alpha_{31}}{\mtx{R}^\alpha_{33}} \vtwo{\sss^\alpha_1}{\sss^\alpha_3} = \vtwo{\ww^\alpha_1}{\ww^\alpha_3}; \qquad
 \mtwo{\mtx{R}^\beta_{22}}{\mtx{R}^\beta_{23}}{\mtx{R}^\beta_{32}}{\mtx{R}^\beta_{33}} \vtwo{\sss^\beta_2}{\sss^\beta_3} = \vtwo{\ww^\beta_1}{\ww^\beta_3}.
\label{eq:doubsys}
\ee
Since the normals of the two leaf boxes are opposed on the interior ``3'' edge,
$\ww^\alpha_3 = -\sss^\beta_3$ and $\sss^\alpha_3 = -\ww^\beta_3$ (using the definitions \eqref{f}-\eqref{g}).
This allows the bottom row equations to be rewritten using only $\alpha$ on the interior edge, namely
\be
\mtx{R}^\alpha_{31}\sss^\alpha_1+\mtx{R}^\alpha_{33}\sss^\alpha_3 = \ww^\alpha_3
\label{eq:fromalpha}
\ee
and
\be {\mtx{R}^\beta_{32}}\sss^\beta_2-{\mtx{R}^\beta_{33}}\ww^\alpha_3 = -\sss^\alpha_3.
 \label{eq:frombeta}
\ee

Let $\mtx{W} := \left(\mtx{I}-{\mtx{R}^\beta_{33}}\mtx{R}^\alpha_{33}\right)^{-1}$.
Plugging (\ref{eq:fromalpha}) into (\ref{eq:frombeta}) results in the following equation
$$ {\mtx{R}^\beta_{32}}\sss^\beta_2-{\mtx{R}^\beta_{33}}\left(\mtx{R}^\alpha_{31}\sss^\alpha_1+\mtx{R}^\alpha_{33}\sss^\alpha_3\right) = -\sss^\alpha_3.$$
By collecting like terms and solving for $\sss^\alpha_3$, we find
\be
\begin{array}{rl}
\sss^\alpha_3 &=  \mtx{W}\left(\mtx{R}^\beta_{33}\mtx{R}^\alpha_{31}\sss^\alpha_1- {\mtx{R}^\beta_{32}}\sss^\beta_2\right)\\
             & = \left[\mtx{W}\mtx{R}^\beta_{33}\mtx{R}^\alpha_{31}\;\;-\!\mtx{W}{\mtx{R}^\beta_{32}}\right]\vtwo{\sss^\alpha_1}{\sss^\beta_2}.
\end{array}
\label{eq:solutionf}
\ee
Note that the matrix $\mtx{S}^\alpha :=  \left[\mtx{W}\mtx{R}^\beta_{33}\mtx{R}^\alpha_{31}\quad -\!\mtx{W} {\mtx{R}^\beta_{32}}\right]$
maps the incoming impedance data on $\Omega^\tau$ to the incoming (with respect to $\alpha$) impedance data on the interior edge. By combining
the relationship between the impedance boundary data on neighbor boxes and (\ref{eq:fromalpha}), the matrix
$\mtx{S}^\beta = -\left[\itim^\alpha_{33}+\mtx{W}\mtx{R}^\beta_{33}\mtx{R}^\alpha_{31}\quad -\!\mtx{W} {\mtx{R}^\beta_{32}}\right]$
computes the impedance data $\sss^\beta_3$.

The outgoing impedance data $\ww^\alpha_3$ is found by plugging $\sss^\alpha_3$ into equation (\ref{eq:fromalpha}).
Now the top row equations of (\ref{eq:doubsys}) can be rewritten without reference to the interior edge.
The top row equation of (\ref{eq:doubsys}) from $\Omega^\alpha$ is now
$$\left({\mtx{R}^\alpha_{11}}+{\mtx{R}^\alpha_{13}}\mtx{W}\mtx{R}^\beta_{33}\mtx{R}^\alpha_{31}\right)\sss^\alpha_1-{\mtx{R}^\alpha_{13}}\mtx{Q}{\mtx{R}^\beta_{32}}\sss^\beta_2 = \ww^\alpha_1$$
and the top row equation of (\ref{eq:doubsys}) from $\Omega^\beta$ is
$$\left(\mtx{R}^\beta_{22}+\mtx{R}_{23}^\beta\mtx{R}^\alpha_{33}\mtx{W}\mtx{R}_{32}^\beta\right)\sss^\beta_2 -
\mtx{R}_{23}^\beta\left(\mtx{R}^\alpha_{31}+\mtx{R}^\alpha_{33}\mtx{W}\mtx{R}^\alpha_{33}\mtx{R}_{31}^\alpha\right)\sss_1^\alpha = \ww^\beta_2.$$
Writing these equations as a system, we find $\sss^\alpha_1$ and $\sss^\beta_2$ satisfy
\begin{equation}\mtwo{{\mtx{R}^\alpha_{11}}+{\mtx{R}^\alpha_{13}}\mtx{W}\mtx{R}^\beta_{33}\mtx{R}^\alpha_{31}}{-{\mtx{R}^\alpha_{13}}\mtx{W}{\mtx{R}^\beta_{32}}}
{-\mtx{R}_{23}^\beta\left(\mtx{R}^\alpha_{31}+\mtx{R}^\alpha_{33}\mtx{W}\mtx{R}^\alpha_{33}\mtx{R}_{31}^\alpha\right)}{\quad\mtx{R}^\beta_{22}+\mtx{R}_{23}^\beta\mtx{R}^\alpha_{33}\mtx{Q}\mtx{R}_{32}^\beta}
\vtwo{\sss^\alpha_1}{\sss^\beta_2} = \vtwo{\ww^\alpha_1}{\ww^\beta_2}.
\label{eq:R_merge}\end{equation}
Thus the block matrix on the left hand side of (\ref{eq:R_merge}) is $\itim^\tau$, the ItI operator for $\Omega^\tau$.


\begin{remark}
In practice, the matrix products $\mtx{W}\mtx{R}^\beta_{33}\mtx{R}^\alpha_{31}$ and $\mtx{W}\mtx{R}^\beta_{32}$
should be computed once per merge.
\end{remark}

\begin{remark}
Note that the formula \eqref{eq:R_merge} is quite different from that for
merging DtN maps appearing in prior work \cite{ONspectralcomposite,2012_martinsson_spectralcomposite}.
The root cause is the way the equivalence of
incoming and outgoing data on the interior edge differs from the case
of Dirichlet and Neumann data.
\end{remark}

Algorithm~1 outlines the implementation of the hierarchical construction of the
impedance operators, by repeated application of the above merge operation.
Algorithm~2 illustrates the downwards sweep to construct
from incoming impedance data $f$
the solution at all Chebyshev discretization points in $\Omega$.
Note that the latter requires the solution matrices $\mtx{S}$ at each level,
and $\mtx{Y}$ for each of the leaf boxes, that were precomputed by Algorithm~1.

The resulting approximation to the top-level ItI operator
$\itim = \itim^1$ is a square matrix which acts on incoming impedance
data living on the total of $4\Ng 2^M$ composite Gauss nodes on $\pO$.
An approximation $\mtx{T}_{\rm int}$ to the interior DtN map
on these same nodes
now comes from inverting equation (\ref{eq:ItI}) for $T_{\rm int}$, to give
\begin{equation}
\mtx{T}_{\rm int} = -{\rm i}\eta\left(\mtx{R}-\mtx{I}\right)^{-1}\left(\mtx{R}+\mtx{I}\right)
~,
\label{eq:getTback}
\end{equation}
where $\mtx{I}$ is the identity matrix of size $4\Ng 2^M$.
The need for conversion from the ItI to the DtN for the domain $\Omega$ will become
clear in the next section.

\begin{remark}
Due to the pseudo-differential order $+1$ of $T_{\rm int}$, we expect the norm of
$\mtx{T}_{\rm int}$ to grow linearly in the number of boundary nodes.
However, it is also possible that $\kappa$ falls close enough to a resonant wavenumber of $\Omega$
that the norm of $\mtx{T}_{\rm int}$ is actually much larger,
resulting in a loss of accuracy due to the inversion of the nearly-singular
matrix $\mtx{R} -\mtx{I}$ in (\ref{eq:getTback}).
In our extensive numerical experiments,
this latter problem has never happened.
However, it is important to include a condition number check
when formula (\ref{eq:getTback}) is evaluated. Should there be a problem, it would be
a simple matter to modify slightly the domain to avoid a resonance. For instance, one
can add a column of leaf boxes to the side of $\Omega$, and then inexpensively update the
computed ItI operator for $\Omega$ to the enlarged domain.
\label{re:omegaRes}
\end{remark}

\begin{figure} 
\setlength{\unitlength}{1mm}
\begin{picture}(95,50)
\put(05,00){\includegraphics[height=50mm]{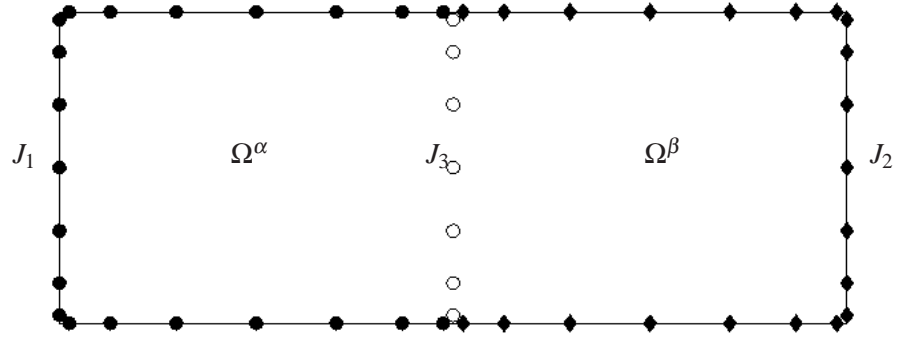}}
\put(32,25){$\Omega^{\alpha}$}
\put(87,25){$\Omega^{\beta}$}
\put(03,25){$J_{1}$}
\put(117,25){$J_{2}$}
\put(58,25){$J_{3}$}
\end{picture}
\caption{Notation for the merge operation described in Section \ref{sec:merge}.
The rectangular domain $\Omega^\tau$ is the union of
two squares $\Omega^{\alpha}$ and $\Omega^{\beta}$.
The sets $J_{1}$ and $J_{2}$ form the exterior Gauss nodes (black), while
$J_{4}$ consists of the interior Gauss nodes (white). (An unrealistically
small node number $\Ng=7$ is used for visual clarity.)
\label{fig:siblings_notation}
}
\end{figure}

\begin{figure}[ht]
\fbox{
\begin{minipage}{140mm}
\begin{center}
\textsc{Algorithm 1} (build solution matrices)
\end{center}

This algorithm builds the global Dirichlet-to-Neumann matrix $\mtx{T}_{\rm int}$ for (\ref{pdep})-(\ref{dir}).
For $\tau$ a leaf box, the algorithm builds the solution matrix $\mtx{Y}^{\tau}$ that maps
impedance data at Gauss nodes to the solution at the interior Chebyshev nodes.
For non-leaf boxes $\Omega^\tau$, it builds the matrices $\mtx{S}^{\tau}$ required for constructing $\sss$ impedance data
on interior Gauss nodes.
It is assumed that if box $\Omega^\tau$ is a parent of box $\Omega^\sigma$, then $\tau < \sigma$.

\rule{\textwidth}{0.5pt}

\begin{tabbing}
\mbox{}\hspace{7mm} \= \mbox{}\hspace{6mm} \= \mbox{}\hspace{6mm} \= \mbox{}\hspace{6mm} \= \mbox{}\hspace{6mm} \= \kill
(1)\> \textbf{for} $\tau = N_{\rm boxes},\,N_{\rm boxes}-1,\,N_{\rm boxes}-2,\,\dots,\,1$\\
(2)\> \> \textbf{if} ($\Omega^\tau$ is a leaf)\\
(3)\> \> \> Construct $\itim^{\tau}$ and $\mtx{Y}^\tau$ via the process described in Section \ref{sec:box}.\\
(4)\> \> \textbf{else}\\
(5)\> \> \> Let $\Omega^\alpha$ and $\Omega^\beta$ be the children of $\Omega^\tau$.\\
(6)\> \> \> Split $J_b^{\alpha}$ and $J_b^{\beta}$ into vectors $J_{1}$, $J_{2}$, and $J_{3}$ as shown in Figure \ref{fig:siblings_notation}.\\
(7)\> \> \> $\mtx{W} = \left(\mtx{I}-{\mtx{R}^{\beta}_{33}}\mtx{R}^{\alpha}_{33}\right)^{-1}$\\
(8)\> \> \> $\itim^\tau = \mtwo{{\mtx{R}^{\alpha}_{11}}+{\mtx{R}^{\alpha}_{13}}\mtx{W}\mtx{R}^{\beta}_{33}\mtx{R}^{\alpha}_{31}}{-{\mtx{R}^{\alpha}_{13}}\mtx{W}{\mtx{R}^{\beta}_{32}}}
{-\mtx{R}_{23}^{\beta}\left(\mtx{R}^{\alpha}_{31}+\mtx{R}^{\alpha}_{33}\mtx{W}\mtx{R}^{\alpha}_{33}\mtx{R}_{31}^{\alpha}\right)}{\quad\mtx{R}^{\beta}_{22}+\mtx{R}_{23}^{\beta}\mtx{R}^{\alpha}_{33}\mtx{Q}\mtx{R}_{32}^{\beta}}$\\
(9)\> \> \> $\mtx{S}^{\alpha} =  \left[\mtx{W}\mtx{R}^{\beta}_{33}\mtx{R}^{\alpha}_{31}\quad -\!\mtx{W} {\mtx{R}^{\beta}_{32}}\right]$.\\
(10)\> \> \> $\mtx{S}^{\beta} = -\left[\itim^{\alpha}_{33}+\mtx{W}\mtx{R}^{\beta}_{33}\mtx{R}^{\alpha}_{31}\quad -\!\mtx{W} {\mtx{R}^{\beta}_{32}}\right]$.\\
(10)\> \> \> Delete $\itim^{\alpha}$ and $\itim^{\beta}$.\\
(11)\> \> \textbf{end if}\\
(12)\> \textbf{end for}\\
(13)\> $\dtnm_{\rm int} = -{\rm i}\eta(\itim^1-\mtx{I})^{-1}(\itim^1+\mtx{I})$
\end{tabbing}
\end{minipage}}
\end{figure}

\begin{figure}[ht]
\fbox{
\begin{minipage}{140mm}
\begin{center}
\textsc{Algorithm 2} (solve BVP (\ref{pder})-(\ref{r}) once solution matrices have been built)
\end{center}

This program constructs an approximation $\uu$ to the solution $u$ of (\ref{pder})-(\ref{r}).
It assumes that all matrices $\mtx{S}^{\tau},\mtx{Y}^\tau$
have already been constructed.
It is assumed that if box $\Omega^\tau$ is a parent of box $\Omega^\sigma$, then $\tau < \sigma$. We call $J^\tau$ the indices of nodes in box $\Omega^\tau$.

\rule{\textwidth}{0.5pt}

\begin{tabbing}
\mbox{}\hspace{7mm} \= \mbox{}\hspace{6mm} \= \mbox{}\hspace{6mm} \= \mbox{}\hspace{6mm} \= \mbox{}\hspace{6mm} \= \kill
(1)\> \textbf{for} $\tau = 1,\,2,\,3,\,\dots,\,N_{\rm boxes}$\\
(2)\> \> \textbf{if} ($\Omega^\tau$ is a leaf)\\
(3)\> \> \>$\uu(J^\tau) = \mtx{Y}^{\tau}\,\sss^\tau$.\\
(4)\> \> \textbf{else} \\
(5)\> \> \> Let $\Omega^\alpha$ and $\Omega^\beta$ be the children of $\Omega^\tau$.\\
(6)\> \> \>$\sss^{\alpha}_3 = \mtx{S}^{\alpha}\sss^\tau$, $\sss^{\beta}_3 = \mtx{S}^{\beta}\sss^\tau$.\\
(7)\>\> \textbf{end if}\\
(8)\> \textbf{end for}
\end{tabbing}
\end{minipage}}
\end{figure}

\section{Well-conditioned boundary integral formulation for scattering}
\label{sec:2ndkind}


In this section, we present an improved
boundary integral equation alternative to the scattering formulation (\ref{eq:intro})
from the introduction.


\subsection{Formula for the exterior DtN operator $T_{\rm ext}$}
\label{sec:formulatText}

We first construct the exterior DtN operator $T_{\rm ext}$
via potential theory. The starting point is
Green's exterior representation formula \cite[Thm.~2.5]{coltonkress},
which states that
any radiative solution $\us$ to the Helmholtz equation in $\Omega^c$ may be written,
\begin{equation}
\us(\xx) = \left(\mathcal{D}\us|_\pO \right)(\xx) - \left(\mathcal{S} \us_n\right)(\xx),
\qquad\mbox{for}\ \xx\in \Omega^c,
\label{eq:greens}
\end{equation}
where $\left(\mathcal{D}\phi\right)(\xx) := \int_{\pO}\frac{\partial}{\partial n_{\yy}} \bigl(\frac{i}{4}H^{(1)}_0(\kappa|\xx-\yy|)\bigr)\phi(\yy)ds_{\yy}$ and
$\left(\mathcal{S}\phi\right)(\xx) :=\int_{\pO} \frac{i}{4}H^{(1)}_0(\kappa|\xx-\yy|) \phi(\yy)ds_{\yy}$
are respectively the frequency-$\kappa$ Helmholtz
double- and single-layer potentials with density $\phi$,
with $H^{(1)}_0$ the outgoing Hankel function of order zero. ,
The vector $n_{\yy}$ denotes the outward unit normal at $\yy\in\pO$,

Letting $\xx$ approach $\pO$ in (\ref{eq:greens}), one finds via the jump relations,
\begin{equation}
\label{eq:victorinox}
\half\us(\xx) = \left(D \us|_{\pO}\right)(\xx) - \left(S \us_n\right)(\xx),\qquad \xx \in \pO,
\end{equation}
where $D$ and $S$ are the double- and single-layer boundary integral operators on $\pO$.
See \cite[Ch.~3.1]{coltonkress} for an introduction to these representations and operators.
Rearranging (\ref{eq:victorinox}) gives
$ \us_n = S^{-1} \left(D-\half I\right)\us|_\pO$,
thus the exterior DtN operator is given in terms of the operators of potential theory by
\be
T_{\rm ext} = S^{-1} \left(D-\half I\right)
~.
\label{Text}
\ee

\subsection{The new integral formulation}
We apply from the left the single layer integral operator $S$
to both sides of \eqref{eq:intro}, and use \eqref{Text}, to obtain
\be
\left(\half I -D +ST_{\rm int}\right) \us|_\pO = S\left(\ui_n - T_{\rm int} \ui|_{\pO}\right)
~,
\label{eq:int2}
\ee
a linear
equation for $\us|_\pO$, the restriction of the scattered wave to the domain boundary $\pO$.
Let
\be
A:=\half I -D +ST_{\rm int}
\label{A}
\ee
be the boundary integral operator appearing in the above formulation.
In the trivial case $b\equiv 0$ (no scattering potential)
it is easy to check that $A = I$, by using
$T_{\rm int} = S^{-1}(D+\half I)$ which can be derived in this case similarly to \eqref{Text}.
Now we prove that introducing a general scattering potential perturbs $A$ only compactly,
that is, our left-regularization of the original ill-conditioned \eqref{eq:intro}
has produced a well-conditioned equation.
\begin{thm} 
Let $\Omega\subset\R^2$ be a bounded Lipschitz domain
containing the support of a bounded scattering potential function $b$.
Let $\kappa>0$ not be a resonant wavenumber of $\Omega$.
Then the operator \eqref{A} takes the form $$A = I + K$$
where $K:L^2(\pO)\to L^2(\pO)$ is compact; thus the integral formulation \eqref{eq:int2}
is of Fredholm second kind.
\label{t:2ndkind}
\end{thm}  
\begin{proof}  
Let $u$ satisfy \eqref{pdep} in $\Omega$, then
by Green's interior representation formula (third identity) \cite[Eq.~(2.4)]{coltonkress},
\be
u(\xx) = \left(\mathcal{S} u_n\right)(\xx) - \left(\mathcal{D}u|_\pO \right)(\xx)
- \kappa^2 (\mathcal{V}\,bu)(\xx)
\qquad\mbox{for}\ \xx\in \Omega
~,
\label{G3I}
\ee
where
$(\mathcal{V} \phi)(\xx) :=
\int_\Omega \frac{i}{4} H^{(1)}_0(\kappa|\xx-\yy|) \phi(\yy) d\yy$ denotes
the Helmholtz volume potential \cite[Sec.~8.2]{coltonkress} acting on a function $\phi$
with support in $\Omega$.
Define $P$ to be the solution operator for the interior Dirichlet problem
\eqref{pdep}-\eqref{dir}, i.e.\ $u(\yy) = (P u|_\pO)(\yy)$ for $\yy\in\Omega$,
and let $B$ denote the operator that multiplies a function pointwise by $b(\yy)$,
we can write the last term in \eqref{G3I} as $-\kappa^2 \mathcal{V} B P u|_\pO$.
Taking $\xx$ to $\pO$ from inside in \eqref{G3I}, the jump relations give
\be
\half u(\xx) = \left(S u_n\right)(\xx) - \left(D u|_\pO \right)(\xx)
- \kappa^2 (V B P u|_\pO)(\xx),
\qquad\xx\in \pO
~,
\label{eq:rewrite}
\ee
where $V$ is $\mathcal{V}$ restricted to evaluation on $\pO$.
Recall $u_n = T_{\rm int} u_\pO$.  Plugging this definition into (\ref{eq:rewrite}), we find
$$
ST_{\rm int} = \half I + D +  \kappa^2 V B P.
$$
When substituted into \eqref{A} results in cancellation of the $D$ terms, giving
$$
A = I + \kappa^2 V B P
~.
$$
Now $P: L^2(\pO)\to L^2(\Omega)$ is bounded \cite[Thm.~4.25]{mclean2000}, and $B$ is bounded.
$\mathcal{V}$ is two orders of smoothing: it is bounded from $L^2(\Omega)$ to
$H^2(W)$ for $W$ any bounded domain \cite[Thm.~8.2]{coltonkress}, and thus also bounded to $H^1(W)$.
While the Sobolev trace theorem has certain
restrictions on the order \cite[Thm.~3.38]{mclean2000} for $\pO$ Lipschitz,
the trace operator is bounded from $H^1(\Omega)$ to $H^{1/2}(\pO)$.
Thus $V:L^2(\Omega)\to H^{1/2}(\pO)$ is bounded.
Since $H^{1/2}(\pO)$ compactly imbeds into $L^2(\pO)$ on a Lipschitz boundary
\cite[Thm.~3.27 and p.99]{mclean2000},
the operator $VBP$ is compact and the proof complete.
\end{proof} 

Note that $D$ in \eqref{A} is {\em not compact} when $\pO$ has corners \cite[Sec.~3.5]{coltonkress},
yet the theorem holds with corners since $D$ is canceled in the proof.

Figure \ref{fig:spec} compares the spectrum of the unregularized (\ref{eq:intro}) and
regularized (\ref{eq:int2}) operators, in a simple computational example.
The improvement in the eigenvalue distribution is dramatic:
the spectrum of (\ref{eq:intro}) has
small eigenvalues but extends to large eigenvalues of order $10^5$,
while the spectrum of (\ref{eq:int2}) is tightly clustered around $+1$
with no eigenvalues of magnitude larger than 2.

\begin{figure}
\begin{center}
\begin{tabular}{ccc}
\setlength{\unitlength}{1mm}
\begin{picture}(60,55)
\put(-10,00){\includegraphics[height=55mm]{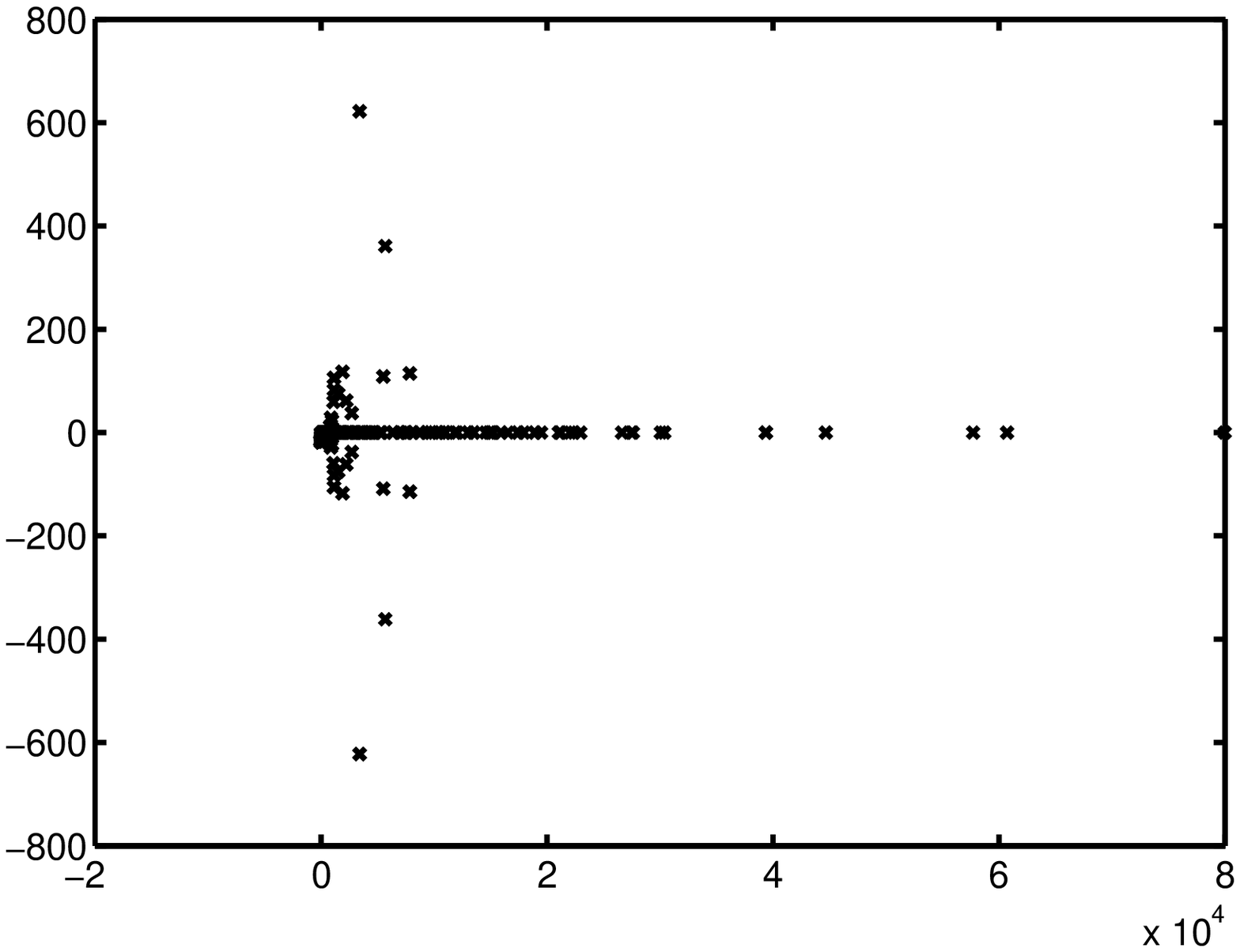}}
\put(-10,24){\rotatebox{90}{$imag(\lambda)$}}
\put(24,00){$real(\lambda)$}
\end{picture}
&  &
\setlength{\unitlength}{1mm}
\begin{picture}(60,55)
\put(-10,00){\includegraphics[height=55mm]{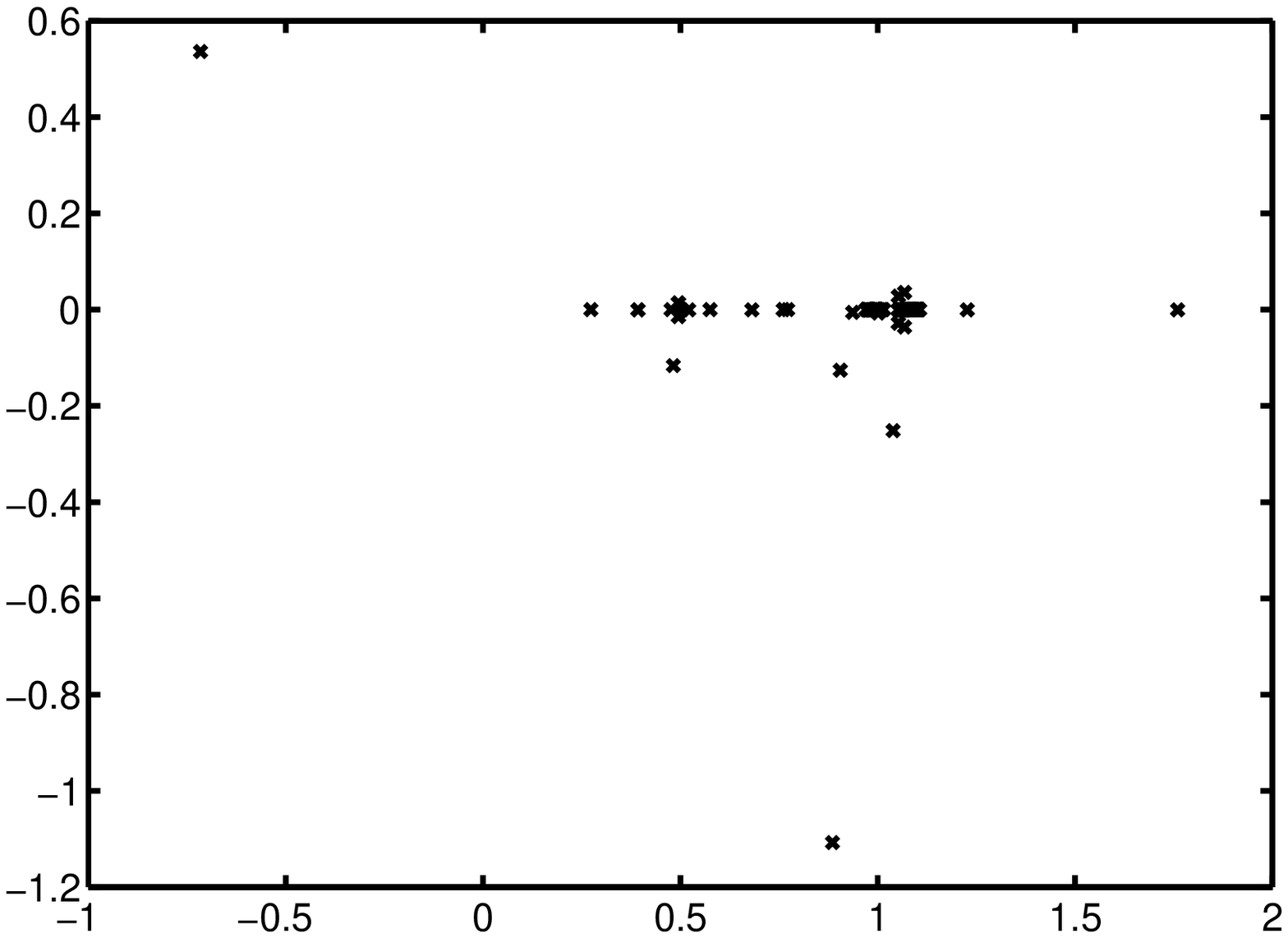}}
\put(-10,24){\rotatebox{90}{$imag(\lambda)$}}
\put(24,00){$real(\lambda)$}
\end{picture}\\
(a) &\mbox{}\hspace{20mm}\mbox{}& (b)
\end{tabular}
\end{center}
\vspace{-1ex}
\caption{Eigenvalues of the discretized operators plotted in the complex plane.
(a) shows the operator $T_{\rm int}-T_{\rm ext}$ from \eqref{eq:intro}, whilst
(b) shows the new regularized operator $A$ from (\ref{eq:int2}).
In both cases $\Omega = (-0.5,0.5)^2$, $\kappa = 20$ and $b(\xx) = -1.5e^{-160|\xx|^2}$.}
\label{fig:spec}
\end{figure}

\subsection{Reconstructing the scattered field on the exterior}
\label{sec:exterior}
Once equation (\ref{eq:int2}) is solved for the scattered wave on $\pO$, the scattered
wave can be found at any point in $\Omega^c$ via the representation (\ref{eq:greens}).
All that is needed is the normal derivative of $\us$ on $\pO$, which
from equation (\ref{eq:bdry1}) is found to be
\begin{equation}
\us_n = T_{\rm int} \left(\ui|_\pO +\us|_\pO\right) -  \ui_n
~.
\label{eq:normw}
\end{equation}
For evaluation of (\ref{eq:greens}),
the native Nystr\"om quadrature on $\pO$ is sufficient for 10-digit accuracy for
all points further away from $\Omega$ than the size of one leaf box;
however, as with any boundary integral method,
for highly accurate evaluation very close to $\pO$ a modified quadrature would be needed.

\subsection{Numerical discretization of the boundary integral equation}
We discretize the BIE (\ref{eq:int2}) on $\pO$ via a Nystr\"om method
with composite (panel-based) quadrature with $n$ nodes in total.
The panels on $\pO$ coincide with the edges of the
leaf boxes from the interior discretization, apart from the eight panels touching corners,
where six levels of dyadic panel refinement are used on each to achieve around 10-digit
accuracy.\footnote{We note that some refinement is necessary even though the solution
is smooth near the (fictitious) corners. However, the extra cost of refinement, as opposed to, say, local corner rounding, is negligible.}
On each of these panels, a 10-point Gaussian rule is used.

For building $n\times n$ matrix approximations to the operators $S$ and $D$ in \eqref{eq:int2},
the plain Nystr\"om method is used for matrix elements corresponding to non-neighboring panels,
while generalized Gaussian quadrature for matrix elements
corresponding to the self- or neighbor-interaction of each panel \cite{gen_quad}.
The matrix $\mtx{T}_{\rm int}$ computed by \eqref{eq:getTback} must also be interpolated
from the $4\Ng2^M$ Gauss nodes on $\pO$ to the $n$ new nodes; since the
panels mostly coincide, this is a local operation analogous to the use
of $\mtx{P}$ and $\mtx{Q}$ matrices in section~\ref{sec:box}.

\section{Computational complexity}
\label{sec:cost}

The computational cost of the solution technique is determined by the cost of constructing the approximate
DtN operator $\dtnm_{\rm int}$ and the cost of solving the boundary integral equation (\ref{eq:int2}).  Let $N$ denote
the total number of discretization points in $\Omega$ required for constructing $\itim$.  As there are $\Nc^2$ Chebyshev points for each
leaf box, the total number of discretization points is roughly $4^M\Nc^2$ (to be precise,
since points are shared on leaf box edges, it is
$N = 4^M(\Nc-1)^2+2^{M+1}(\Nc-1)+1$).   Recall that $n$ is the number of points on
$\pO$ required to solve the integral equation.  Note that $n \sim \sqrt{N}$.


\subsection{Using dense linear algebra}
\label{sec:N1.5}
Using dense linear algebra, the cost of constructing $\itim$ via the technique in section \ref{sec:spec}
is dominated by the cost at the top level where a matrix of size
$\sqrt{N}\times \sqrt{N}$ is inverted.  Thus the computational cost is $O(N^{3/2})$.  The cost
of approximating the DtN operator $\dtnm_{\rm int}$ is also $O(N^{3/2})$.  However,
the computational cost of applying $\dtnm_{\rm int}$ is $O(N)$.  If the solution
in the interior of $\Omega$ is desired, the computation cost of the solve is $O(N)$ as well.

The cost of inverting the linear system resulting from the (eg. Nystr\"om) discretization of (\ref{eq:int2}) is $O(n^3)$.
It is possible to accelerate the solve by using iterative methods such as GMRES,
which, given its second kind nature, would converge in $O(1)$ iterations.

When there are multiple incident waves at the same wavenumber $\kappa$, the solution technique should be separated into
two steps: precomputation and solve.  The precomputation step consist of constructing the approximate ItI, and DtN operators
$\itim$ and $\dtnm_{\rm int}$, respectively.  Also included in the precomputation should
be the discretization and inversion of the BIE (\ref{eq:int2}).  The solve step then consists of applying the
inverse of the system in (\ref{eq:int2}).  The precomputation need only be done once per wavenumber with
a computational cost $O(N^{3/2})$.  The cost of each solve (one for each incident wave) is simply the cost
of applying an $n\times n$ dense matrix $\sim O(N)$.

\subsection{Using fast algorithms}
\label{sec:N}
The matrices $\mtx{R}^{\tau}$ in Algorithm 1 that approximate ItI operators, as well as the matrices
$\mtx{T}_{\rm int}$ and $\mtx{T}_{\rm ext}$ approximating DtN operators, all have internal structure
that could be exploited to accelerate the matrix algebra. Specifically, the off-diagonal blocks of these
matrices tend to have low numerical ranks, which means that they can be represented efficiently in so
called ``data-sparse'' formats such as, e.g., $\mathcal{H}$ or $\mathcal{H}^{2}$-matrices \cite{hackbusch,2010_borm_book,2008_bebendorf_book},
or, even better, the Hierarchically Block Separable (HBS) format
\cite{2012_martinsson_FDS_survey,2012_ho_greengard_fastdirect}
(which is closely related to the ``HSS'' format \cite{2010_gu_xia_HSS}).
If the wavenumber $\kappa$ is kept fixed as $N$ increases, it turns out to be possible to accelerate
all computations in the build stage to optimal $O(N)$ asymptotic complexity, and the solve stage
to optimal $O(N^{1/2})$ complexity, see  \cite{ONspectralcomposite}.
However, the scaling constants suppressed by the big-O notation depend on $\kappa$ in such a way
that the use of accelerated matrix algebra is worthwhile primarily for problems of only moderate size (say a few
dozen wavelengths across). Moreover, for high-order methods such as ours, it is common to keep
the number of discretization nodes per wavelength fixed as $N$ increases (so that $\kappa \sim N^{1/2}$),
and in this environment, the scaling of the ``accelerated'' methods revert to $O(N^{3/2})$ and $O(N)$ for the build and the
solve stages, respectively.



\section{Numerical experiments}
\label{sec:numerics}
This section reports on the performance of the new solution technique
for several choices of potential $b(\xx)$ where the (numerical) support of $b$ is contained
in $\Omega = (-0.5,0.5)^2$.  The incident wave is
a plane wave $\ui(\xx) = e^{i\kappa \www\cdot\xx}$ with incident unit direction vector $\www\in\R^2$.

Firstly, in section \ref{sec:acc} the method is applied to
problems where $b(\xx)$ is a single Gaussian ``bump.''
In this case the radial symmetry
allows for an independent semi-analytic solution,
which we use to verify the accuracy of the method.
Then section \ref{sec:conv} reports on the performance of the method when
applied to more complicated problems.  Finally, section \ref{sec:time}
illustrates the computational cost in practice.

For all the experiments, for the composite spectral method described in section \ref{sec:spec}
we use on each leaf a $\Nc\times\Nc$ Chebyshev tensor product grid with $\Nc=16$,
and the number of Gaussian nodes per side of a leaf is $\Ng=14$.

We implemented the methods based on dense matrix algebra with $O(N^{3/2})$ asymptotic
complexity described in section \ref{sec:N1.5}. (We do not use the $O(N)$ accelerated
techniques of section \ref{sec:N} since we are primarily interested in scatterers that
are large in comparison to the wavelength.)

All experiments were executed on a desktop
workstation with two quad-core Intel Xeon E5-2643 processors and 128 GB
of RAM. All computations were done in MATLAB (version 2012b),
apart from the evaluation of Hankel functions in the Nystr\"om and scattered wave calculations,
which use Fortran.
We expect that, careful implementation of the whole scheme in a compiled language can improve
execution times substantially.

\subsection{Accuracy of the method}
\label{sec:acc}
In this section, we consider problems where the scattering potential $b(\xx)$ is given by a Gaussian bump.
Since $b$ has radial symmetry, we may compute an accurate reference
scattering solution by solving a series of ODEs, as explained in
Appendix~\ref{s:ref}.
With $\kappa = 40$ (so that the square $\Omega$ is around six free-space wavelengths on a side),
and $\www  = (1,0)$, we consider two problems given as follows,

\vspace{.5ex}
\begin{tabular}{rl}
\emph{Bump 1}: &  $b(\xx) = -1.5e^{-160 |\xx|^2}$,\\
\emph{Bump 2}: &  $b(\xx) = 1.5e^{-160 |\xx|^2}$.\\
\end{tabular}
\vspace{.5ex}

For \emph{Bump 1}, the bump region has an increased refractive index,
varying from 1 to around 1.58, which can be interpreted as an attractive potential.
For \emph{Bump 2}, the potential is repulsive, causing the waves to become slightly evanescent
near the origin (here the refractive index decreases to zero then becomes purely imaginary,
but note that this does not correspond to absorption.)
Figure \ref{fig:bumps} illustrates
the geometry and the resulting real part of the total field for each experiment.

\begin{figure}[ht]
 \begin{tabular}{ccc}
\includegraphics[height=45mm]{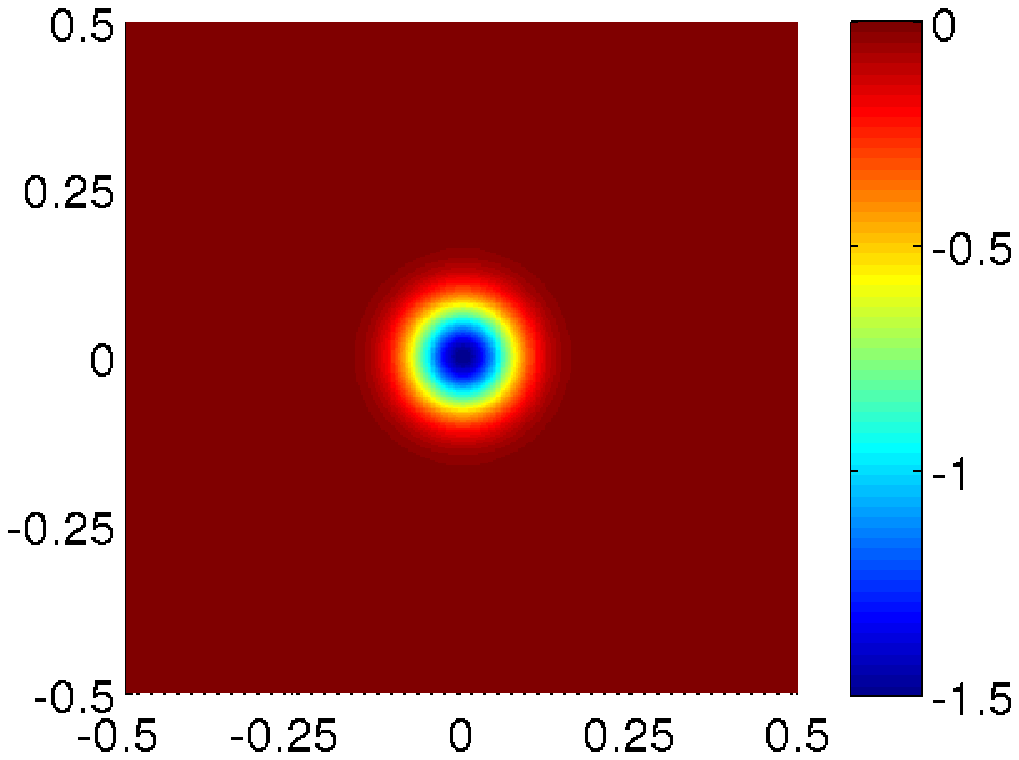}&\mbox{}& \includegraphics[height=45mm]{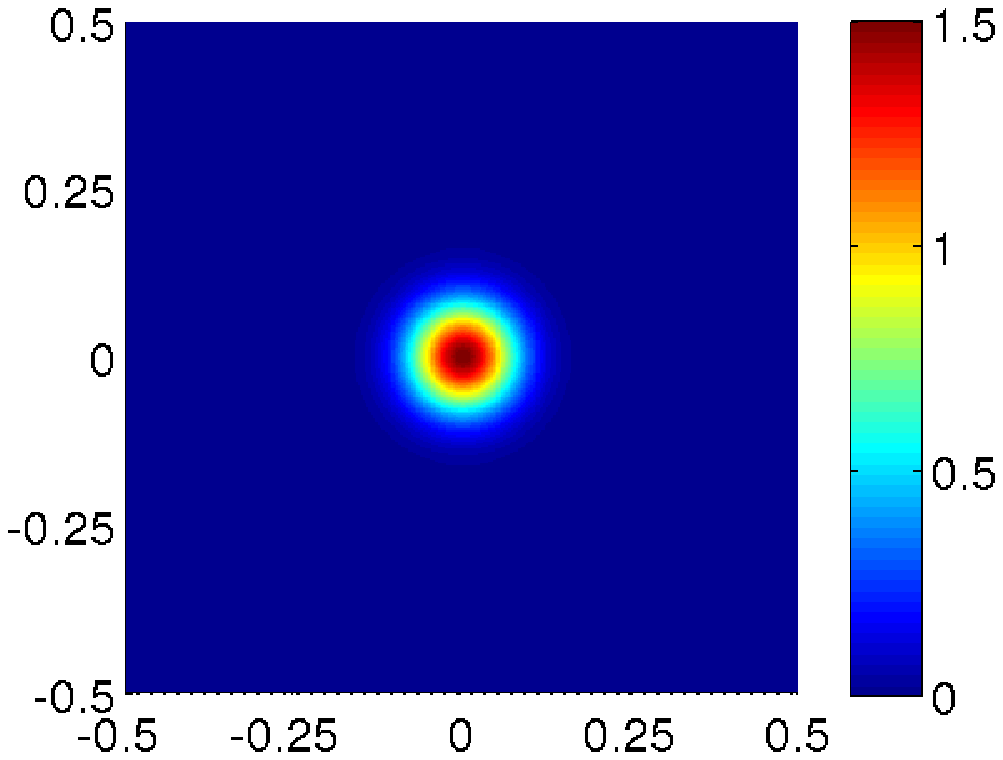}\\
(a)& \mbox{}& (b)\\
 \includegraphics[height=45mm]{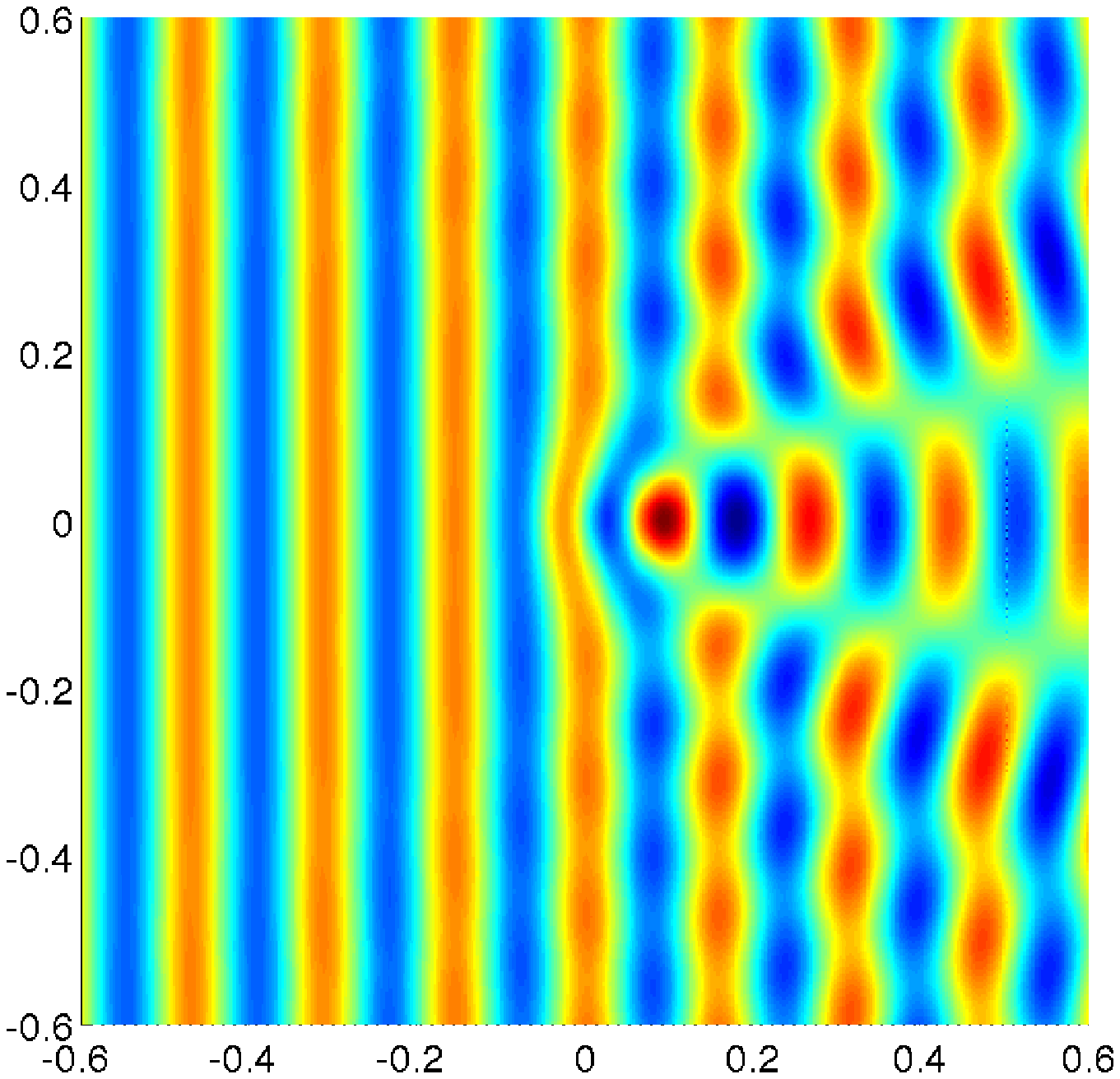}&\mbox{}& \includegraphics[height=45mm]{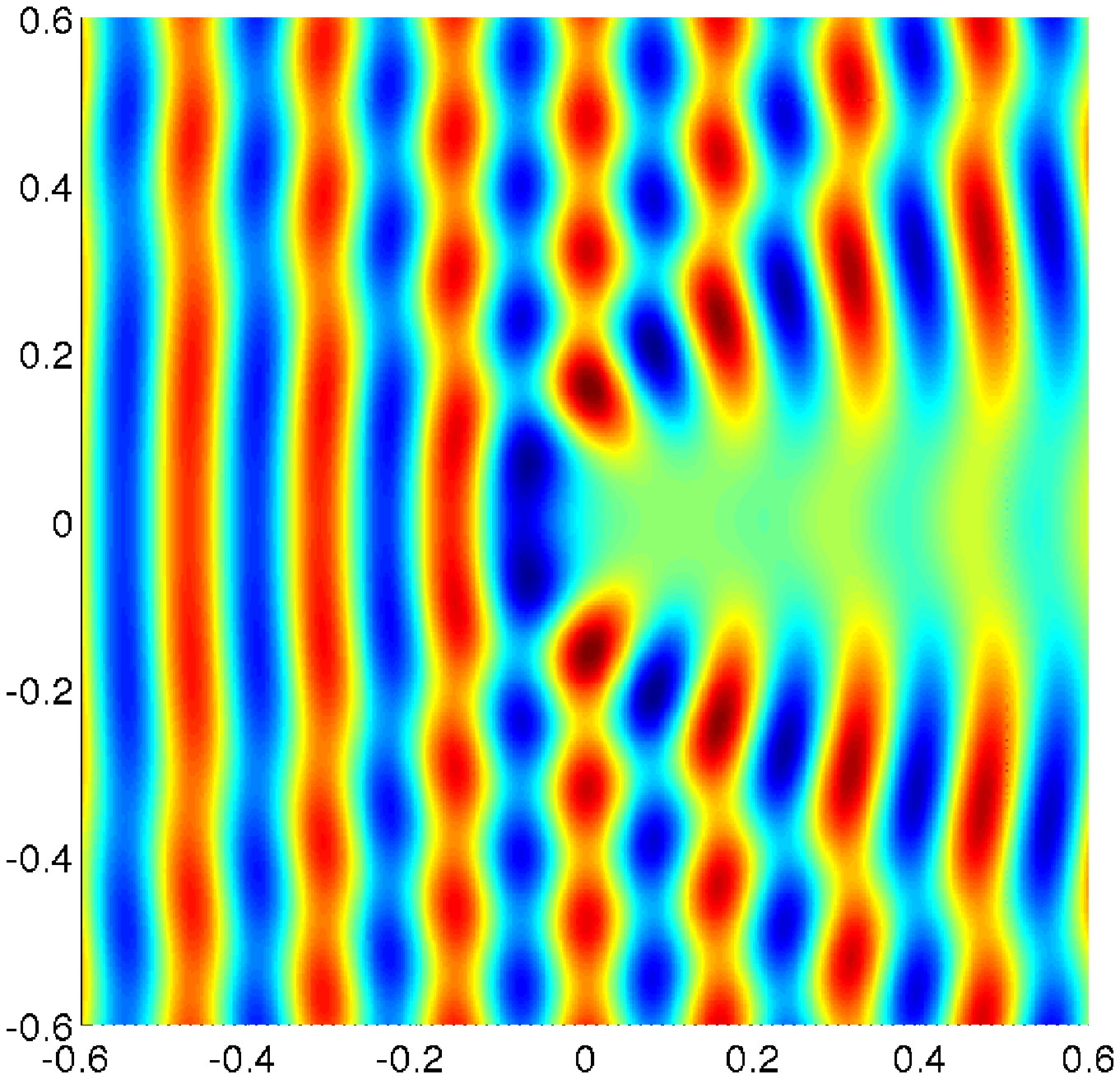}\\
(c)& \mbox{}& (d)\\
\end{tabular}
\caption{\label{fig:bumps} Plots (a) and (b) illustrate $b(\xx)$ for experiments \emph{Bumps 1} and \emph{Bumps 2} in section \ref{sec:acc}.  Plots (c) and (d)
illustrate the real part of the total field for each experiment respectively. }
\end{figure}

Let $\tilde{u}$ denote the approximate total field constructed via the proposed method,
and $u$ denote the reference total field computed as in Appendix~A.
Table \ref{tab:bumps} reports

\vspace{.5ex}
\begin{tabular}{rl}
 $N$: & the number of discretization points used by the composite spectral method in $\Omega$, \\
$n$:& the number of discretization points used for discretizing the BIE,  \\
$\re \tilde{u}(0.5,0)$:& real part of the approximate solution at $(0.5,0)$ (on $\pO$),\\
 $e_1$:& $=|u(0.5,0)-\tilde{u}(0.5,0)|$,\\
$\re \tilde{u}(1,0.5)$:& real part of the approximate solution at $(1,0.5)$ (outside of $\Omega$),\\
$e_2$:&   $=|u(1,0.5)-\tilde{u}(1,0.5)|$.
\end{tabular}
\vspace{.5ex}

In the table, the number of levels $M$ grows from 2 to 5, roughly quadrupling $N$ each time.
High-order convergence is apparent, reaching an accuracy of 9-10 digits
(accuracy does not increase much beyond 10 digits).
At the highest $N$, there are about 50 gridpoints per wavelength at the shortest
wavelength occurring at the center of {\em Bump 1}.

\begin{table}[ht]
 \begin{tabular}{|c|c|c|c|c|c|c|}
\hline
 & $N$& $n$&$\re \tilde{u}(0.5,0)$ &$e_1$ &$\re \tilde{u}(1,0.5)$ &$e_2$ \\ \hline
\multirow{4}{*}{\emph{Bump 1}} &3721&640&-0.98792561833285&7.75e-05&-1.12207402682737&5.09e-05 \\
&14641&800&-0.987981264965721&1.78e-07&-1.12205758254400&8.18e-08 \\
&58081&1120&-0.987981217277174&2.56e-09&-1.12205766387011&1.15e-10\\
&231361&1760&-0.987981215350216&9.31e-10&-1.12205766378840&7.90e-11\\ \hline
\multirow{4}{*}{\emph{Bump 2}} &3721&640&-0.0470314782486572&4.82e-05&-1.01063677552351&3.23e-05\\
&14641&800&-0.0470619180279044&	4.25e-08&-1.01065022958956&7.32e-08\\
&58081&1120&-0.0470619010992819&1.32e-09&-1.01065028579517&1.29e-10\\
&231361&1760& -0.0470619007119554& 5.07e-10&  -1.01065028569638&4.36e-11 \\ \hline
  \end{tabular}
\vspace{2ex}
\caption{\label{tab:bumps} Approximate solutions and pointwise errors for the experiments in section \ref{sec:acc}}
\end{table}

\begin{figure}[ht]
 \begin{tabular}{ccc}
\includegraphics[height=50mm]{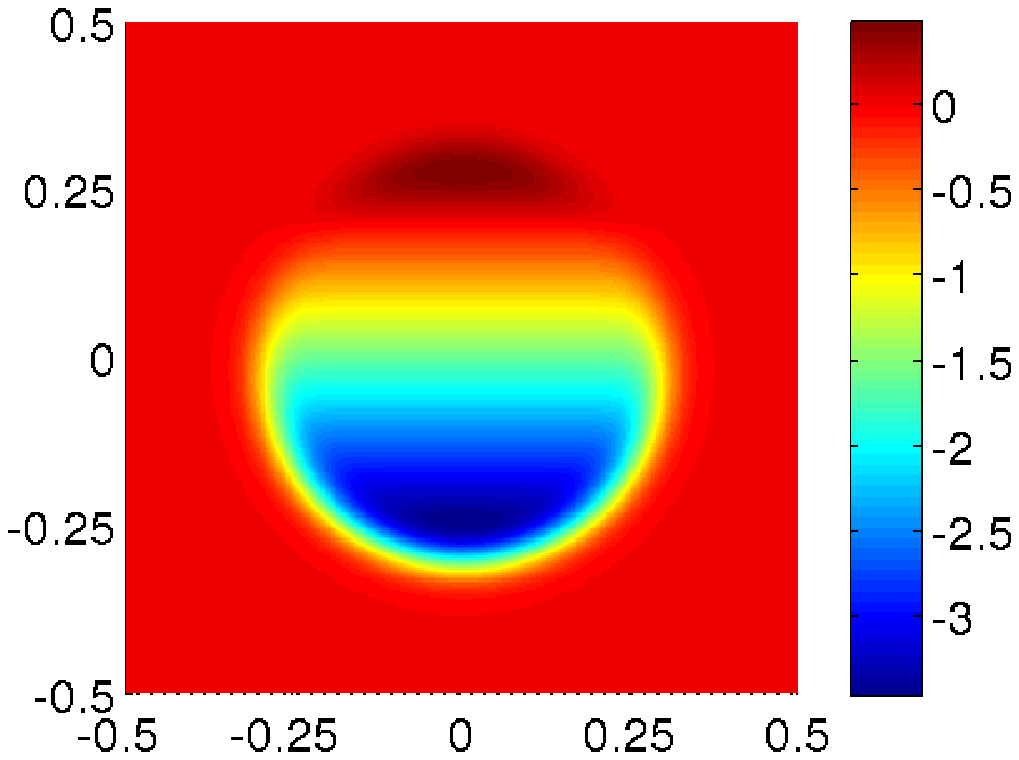}& \includegraphics[height=50mm]{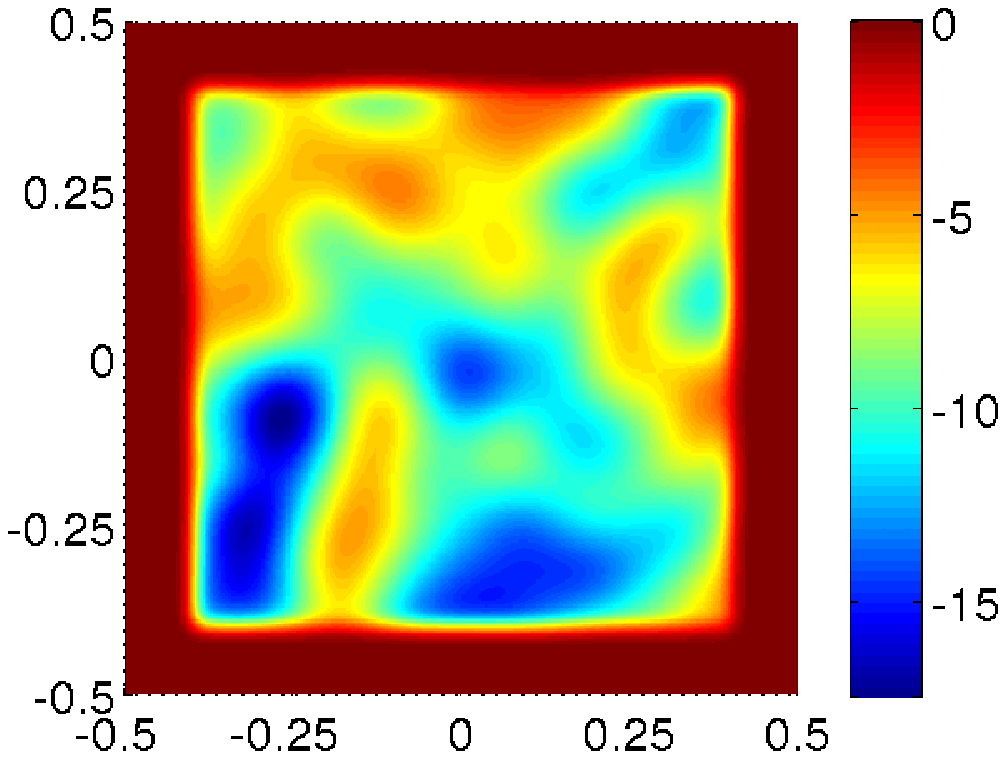}\\
(a)& (b)\\
\multicolumn{2}{c}{\includegraphics[height=50mm]{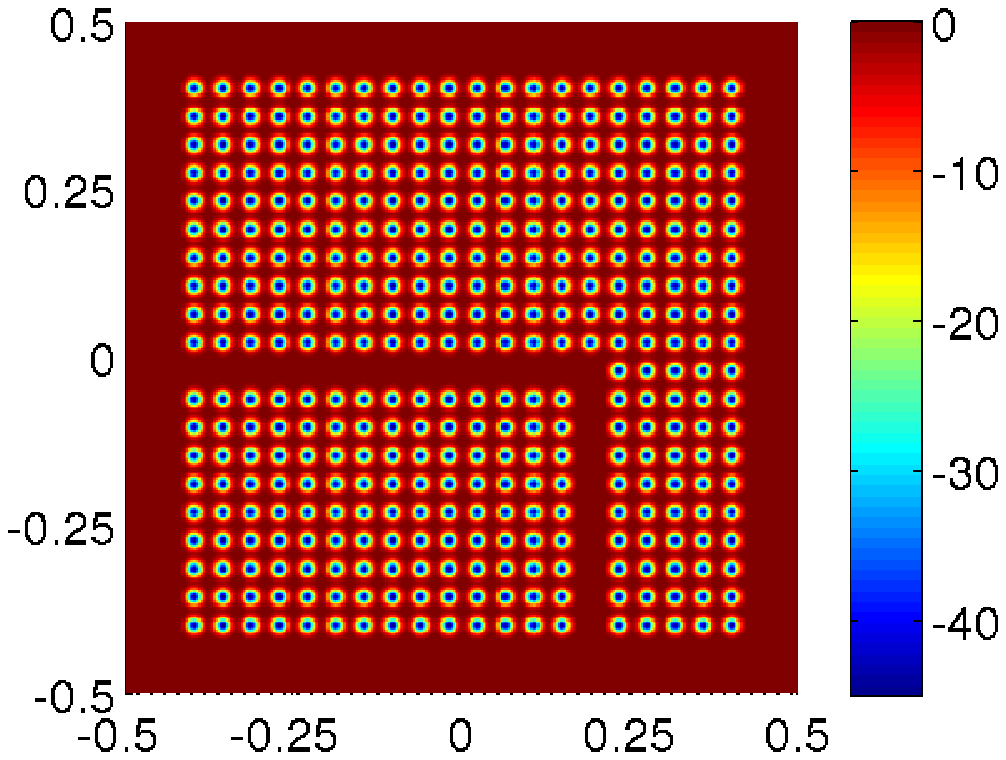}}\\
\multicolumn{2}{c}{(c)}\\
   \end{tabular}
\caption{\label{fig:bs} Plots of the different scattering potentials $b$
inside $\Omega$ in section \ref{sec:conv}.}
\end{figure}

\subsection{Performance for challenging scattering potentials}
\label{sec:conv}
This section illustrates the performance of the numerical method for
problems with smoothly varying wave speed inside of $\Omega$.

We consider three different test cases of scattering potential $b$.  They are

\vspace{.5ex}
\hspace{-5ex}
\begin{tabular}{rl}
 \emph{Lens}: & A vertically-graded lens (Figure \ref{fig:bs}(a)), at wavenumber $\kappa=300$.\\
 & Specifically, $b(\xx) = 4(x_2-0.2)[1-{\rm erf}(25(|\xx|-0.3))]$, where $\xx=(x_1,x_2)$.  \\
& The maximum refractive index is around 2.1\\
\emph{Random bumps}: & The sum of $200$ wide Gaussian bumps randomly placed in $\Omega$,
rolled off to zero\\
&(see Figure \ref{fig:bs}(b))
giving a smooth random potential at wavenumber $\kappa=160$.\\
& The maximum refractive index is around 4.3.
\\
\emph{Photonic crystal}:& $20\times 20$ square array of small Gaussian bumps
(with peak refractive index 6.7) \\
& with a ``waveguide'' channel removed (Figure \ref{fig:bs}(c)).
The wavenumber $\kappa = 77.1$ \\
&is chosen carefully to lie in the first complete bandgap of the crystal.
\end{tabular}
\vspace{.5ex}

For the first two cases, $\Omega$ is around 70 wavelengths on a side, measured using the typical
wavelength occurring in the medium
(for the {\em lens} case, it is 100 wavelengths on a side at the minimum wavelength).
This is quite a high frequency for a variable-medium problem at the accuracies we achieve.
In these two cases the waves mostly {\em propagate}; in contrast, in the third case the waves
mostly {\em resonate} within each small bump,
in such a way that large-scale propagation through the crystal is impossible
(hence evanescent), except in the channel.\footnote{The choice of bump height and width
needs to be made carefully to ensure that a usable bandgap exists;
this was done by creating a separate spectral solver for the band structure of the periodic problem.}

For each choice of varying wave speed, the incident wave is in the direction $\www  = (1,0)$
(we remind the reader that the method works for arbitrary incident direction).
For the \emph{photonic crystal}, we also consider the incident wave direction
$\www = (-\sqrt{2}/2,\sqrt{2}/2)$.
There are no reference solutions available for these problems, hence we study convergence.

In addition to the number of discretization points $N$ and $n$, Table \ref{tab:vary} reports

\vspace{.5ex}
\begin{tabular}{rl}
$\re\tilde{u}(0.25,0)$:& real part of the approximate solution at $(0.25,0)$ (inside of $\Omega$),\\
 $e_1$:& $=|\tilde{u}_N(0.25,0)-\tilde{u}_{4N}(0.25,0)|$, an estimate of the pointwise error,\\
$\re\tilde{u}(1,0.5)$:& real part of the approximate solution at $(1,0.5)$ (outside of $\Omega$),\\
$e_2$:&   $=|\tilde{u}_N(1,0.5)-\tilde{u}_{4N}(1,0.5)|$, an estimate of the pointwise error.\\
\end{tabular}
\vspace{.5ex}

\begin{table}[ht]
 \begin{tabular}{|c|c|c|c|c|c|c|c|}
\hline
&  $N$& $n$&  $\re\tilde{u}(0.5,0)$ &$e_1$ &$\re\tilde{u}(1,0.5)$ &$e_2$\\ \hline
\multirow{5}{*}{Lens}
		    &58081&1120&-0.373405022283892&	2.02e-01	&-0.547735180732198&5.09e-01\\
		    &231361&1760& -0.221345395796661&3.53e-03	&0.161212542340161& 2.86e-03	\\
		    &923521&3040&-0.218651605400620&1.61e-07	& 0.158422280450864& 1.87e-07\\
		    &3690241& 5600&  -0.218651458554288& 6.85e-10	&0.158422464920298&	6.99e-10\\
		    &14753281& 10720&  -0.218651458391577& - & 0.158422464625727 & - \\
\hline
\multirow{5}{*}{Bumps}
		    &58081&1120&1.29105948477323 &     1.91         &-0.612141744074168 &0.44 	\\
		    &231361&1760&0.359271869087464& 	1.99e-02	&-0.931198083868205& 1.87e-02\\
		    &923521&3040&0.374697595070227&	2.76e-06&-0.945752835546445& 7.95e-07\\
		    &3690241& 5600&0.374698518812982&	3.06e-09&-0.945753626496863& 1.58e-09 \\
		    &14753281& 10720& 0.374698518930658& -	&-0.945753627849080 &- \\ \hline

\multirow{5}{*}{Crystal}
		    &58081&1120&-0.406418011063883 &2.03 & -0.129067996215635& 3.42e-01\\
		    &231361&1760&0.0424527158875615&1.53e-03 &0.195870563479998 &1.82e-4\\
		    &923521&3040&0.0437392735790711&7.30e-07	&0.195981633749759& 2.81e-07\\
		    &3690241& 5600&0.0437393320806644&4.84e-10 & 0.195981570696692&  7.56e-10\\
		    &14753281& 10720&0.0437393324622741&- & 0.195981570519668& -\\ \hline
\multirow{2}{*}{Crystal}
		    &58081&1120&-0.0420633119821246 	 &1.28 & -1.20915538109562	& 1.37e-01\\
&231361&1760&0.0367128964251903	&3.47e-03 &-1.09529393341122	 &1.10e-3\\
\multirow{3}{*}{ $\www = \left(-\frac{\sqrt{2}}{2},\frac{\sqrt{2}}{2}\right)$}
&923521&3040&0.0376839447575234	&1.58e-06&-1.09445429347097	& 7.25e-07\\
		    &3690241& 5600&0.0376833752064704	&6.79e-10 &-1.09445431363369 	&  4.90e-9\\
		    &14753281& 10720&0.0376833752704930	&- &-1.09445430874556	& -\\ \hline
 \end{tabular}
\vspace{1ex}
\caption{\label{tab:vary} Convergence results for the experiments in section \ref{sec:conv}.}
\end{table}

Table~\ref{tab:vary} shows that typically 9-digit accuracy is reached
when $N \approx 3.7\times 10^6$ ($M=7$),
which corresponds to 1921 Chebyshev nodes in each direction,
or around 20 nodes per wavelength at the shortest wavelengths in each medium.

\begin{figure}[ht]
\hspace{-5ex}
 \begin{tabular}{ccc}
\includegraphics[height=60mm]{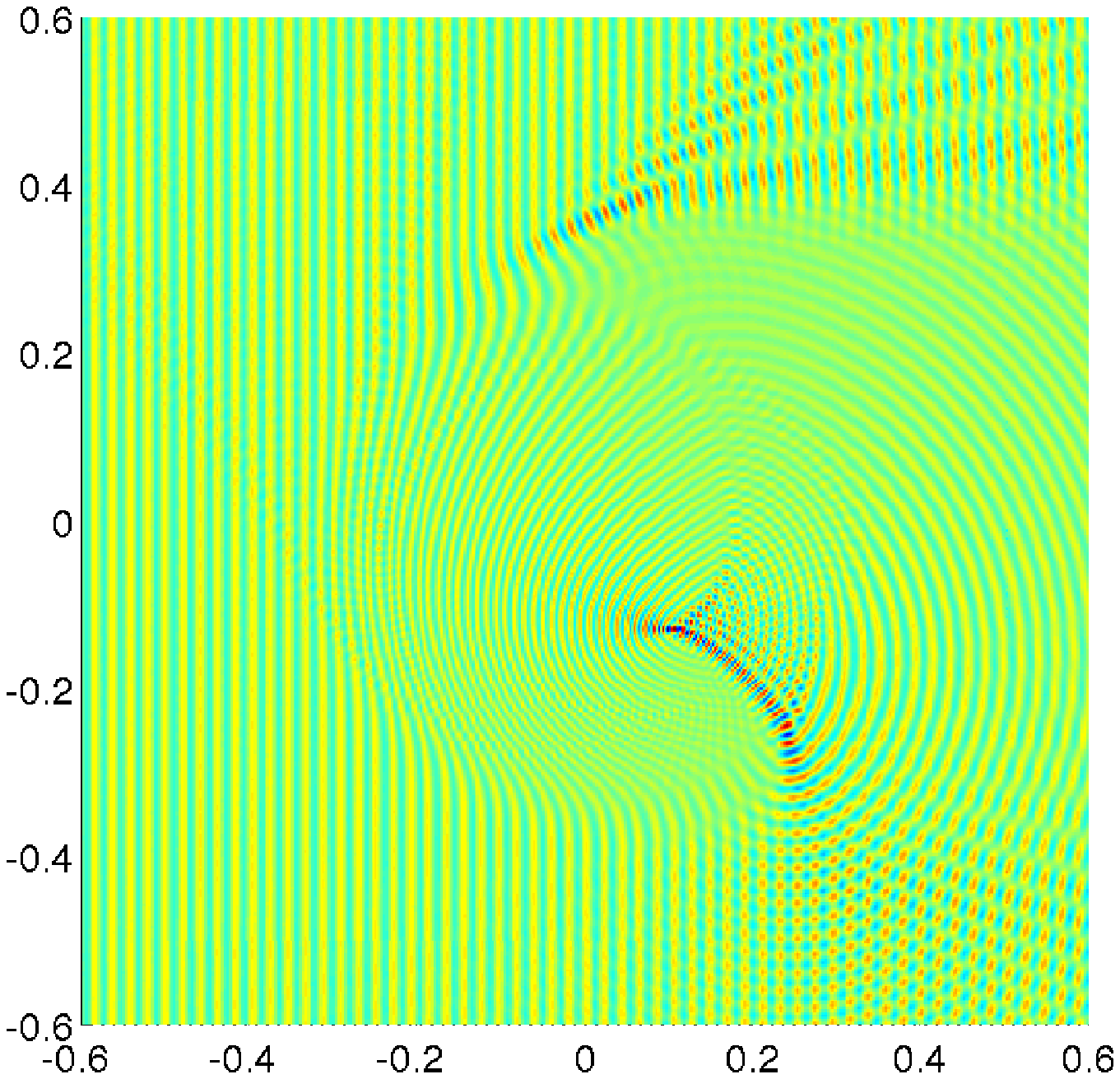}&\mbox{}& \includegraphics[height=60mm]{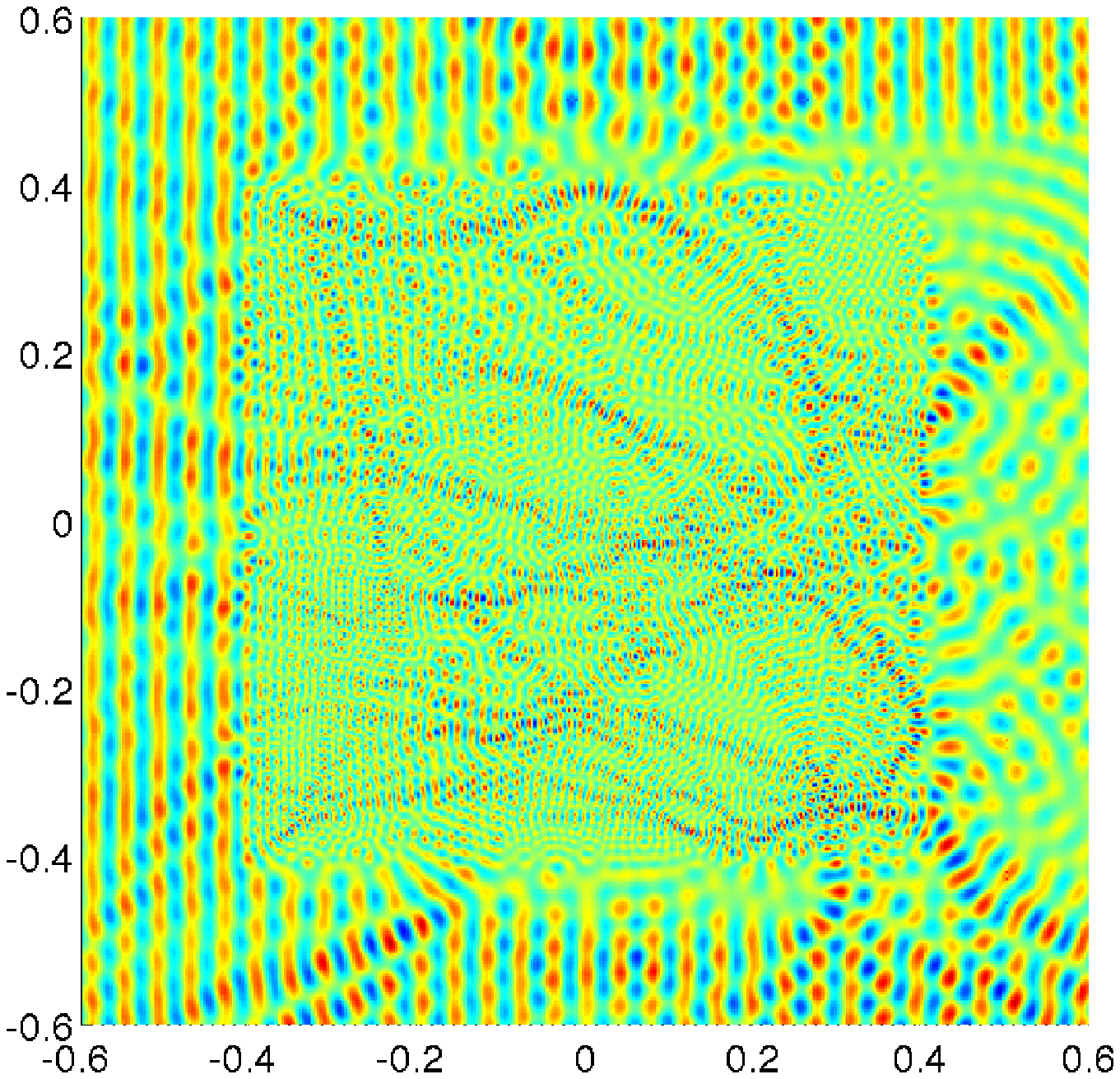}\\
(a)& \mbox{}& (b)\\
\includegraphics[height=60mm]{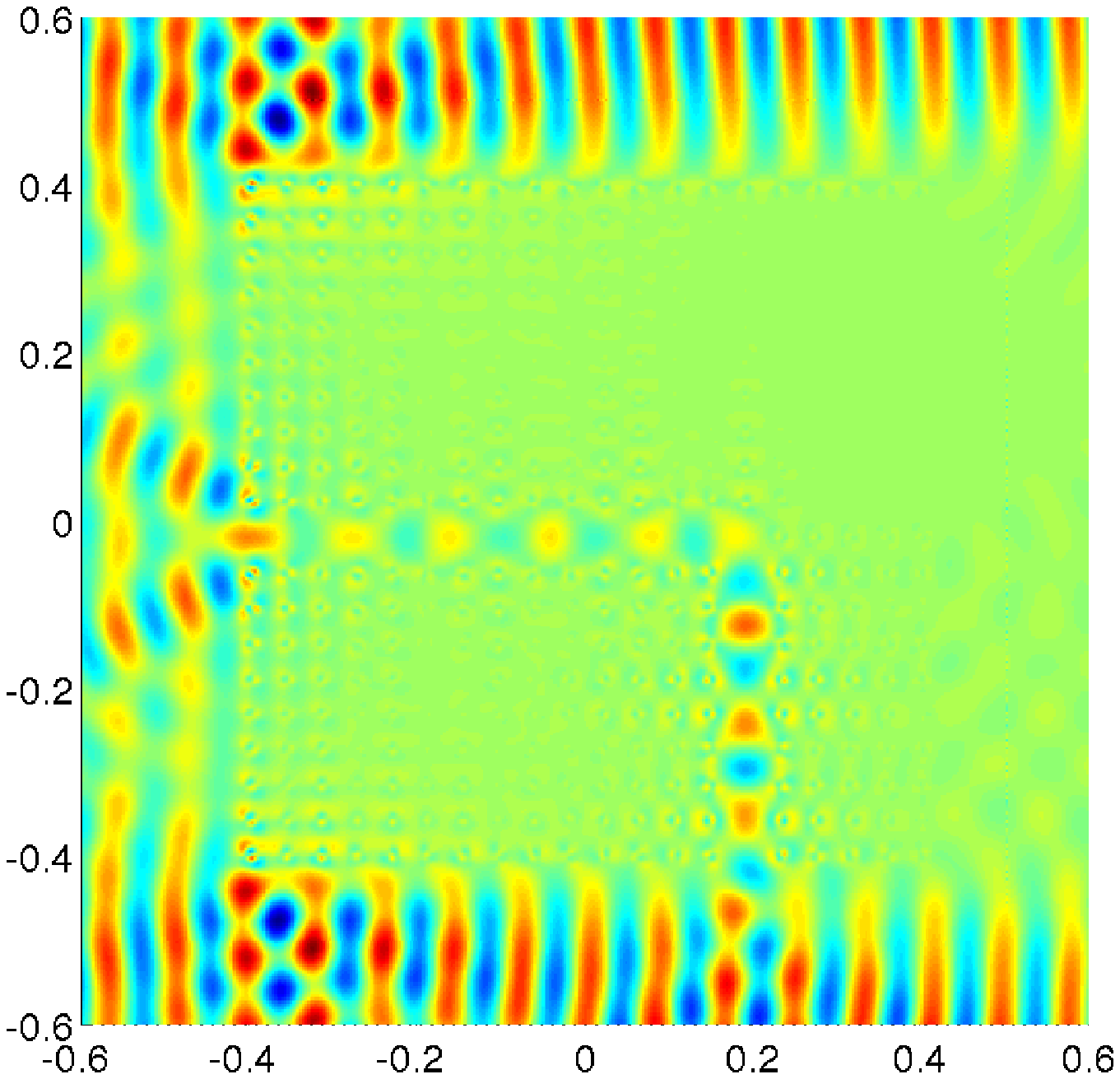}& \mbox{}& \includegraphics[height=60mm]{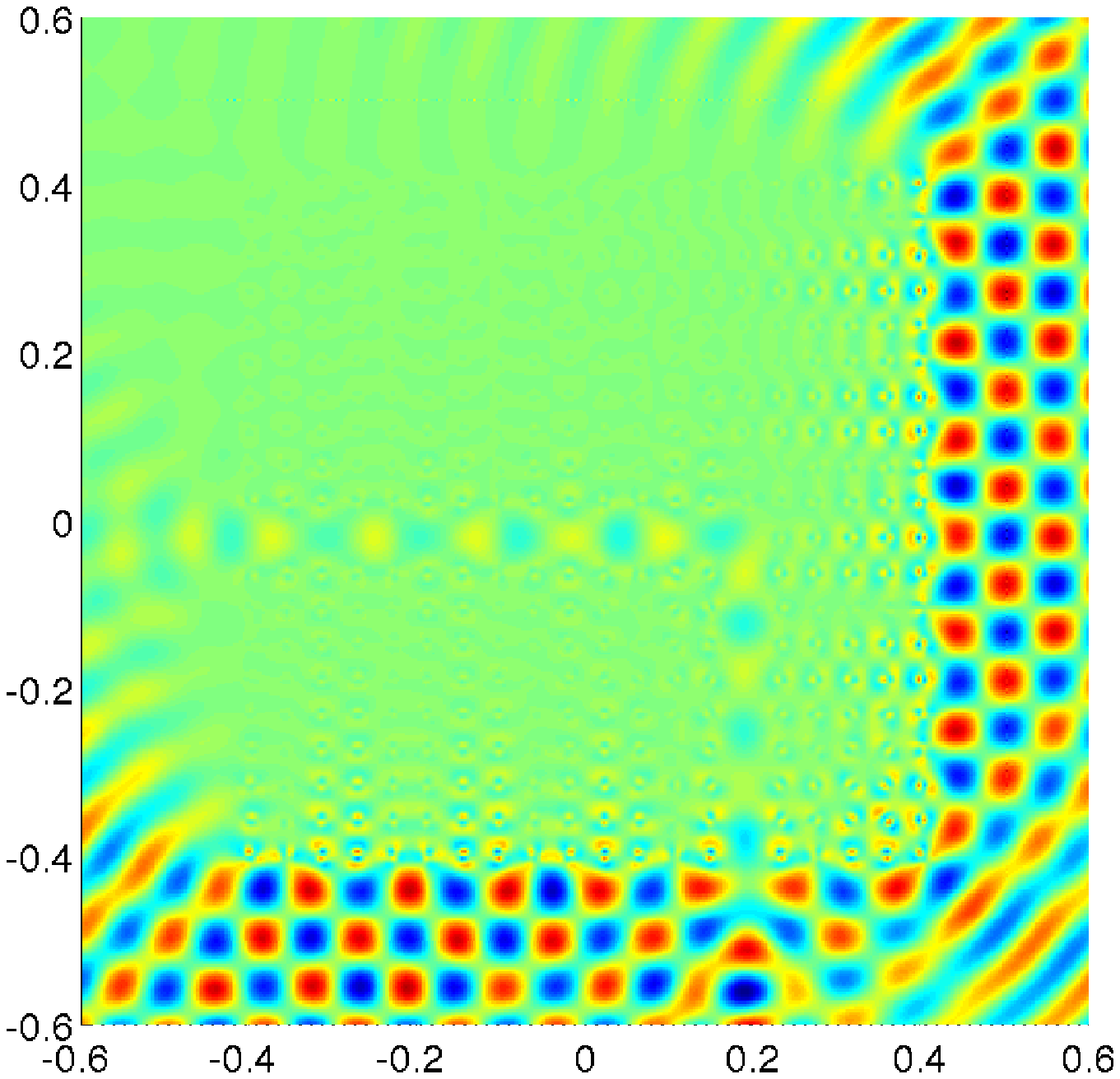}\\
(c) & \mbox{}& (d) \\
   \end{tabular}
\caption{\label{fig:fields}%
Plots of the real part of the total field for the four experiments in section~\ref{sec:conv}
whose scattering potentials are shown in Figure~\ref{fig:bs}:
(a) lens, (b) random bumps, (c) and (d) photonic crystal with different incident wave directions.}
\end{figure}

\subsection{Scaling of the method}
\label{sec:time}
Recall that in the case of multiple incident waves, the solution technique
should be broken into two steps: precomputation and solve.  Since a direct
solver is used, the timing results are independent of the particular scattering potential.

For each choice of $N$ and $n$, Table \ref{tab:times}
reports

\vspace{.5ex}
\begin{tabular}{rl}
$T_{\rm build}$:& Time in seconds to building the approximate ItI and DtN,\\
$T_{\rm solve}$:& Time in seconds to discretize and invert the BIE (\ref{eq:int2}),\\
$T_{\rm apply}$:& Time in seconds to apply the inverse $A^{-1}$
of the discretized integral equation\\
$R_{\rm build}$: & Memory in MB required to store the ItI and solution operators in
the hierarchical scheme,\\
$R_{\rm solve}$: & Memory in MB required to store the discretized inverse $A^{-1}$.
\end{tabular}
\vspace{.5ex}

Figure \ref{fig:times} plots the timings against the problem size $N$.
(The total precomputation time is the sum of $T_{\rm build}$ and $T_{\rm solve}$.)
The results show that even at the largest $N$ tested, the
precomputation time has not reached its asymptotic $O(N^{3/2})$;
the large dense linear algebra has not yet started to dominate $T_{\rm build}$
(this may be due to MATLAB overheads).
However, the cost for solving (\ref{eq:int2}) and applying the inverse scale
closer to expectations.

The memory usage scales as the expected $O(N \log N)$.
We are not able to test beyond 15 million unknowns ($M=8$)
since by that point the memory usage approaches 100 GB.
However, note that if all that is needed is the {\em far-field} solution for
arbitrary incident waves at one wavenumber, the $\mtx{S}^\tau$ and $\mtx{Y}^\tau$ solution
matrices need not be stored, reducing memory significantly, and the final solution
matrix only requires 2 GB.
We note that, extrapolating from the convergence study, this $N$ should be
sufficient for 9-digit accuracy for problems up to 200 wavelengths on a side.

\begin{table}[ht]
\centering
 \begin{tabular}{|c|c|c|c|c|c|c|}
\hline
$N$& $n$ &$T_{\rm build}$ & $T_{\rm solve} $ &$T_{\rm apply}$ & $R_{\rm build}$ & $R_{\rm solve}$\\ \hline
3721&	640&	0.506&	1.78&	5.39e-04 & 9.71 & 6.25\\ 
14641&	800&	0.709&	2.01&	8.28e-04 & 48.07 & 9.77\\ 
58081&	1120&	2.90&	3.01&	1.73e-03 & 229.05& 19.14\\ 
231361&	1760&	12.09&	5.40&	3.32e-03 & 1063.23& 47.37\\ 
923521&	3040&	51.67&	13.23&	1.05e-02 & 4841.01& 141.02\\ 
3690241&5600&	231.18&	40.79&	4.03e-02 & 21716.21& 478.52\\ 
14753281&10720&	1081.09&185.54&	1.13e-01& 96273.17&1753.52\\ \hline
 \end{tabular}
\vspace{2ex}
\caption{\label{tab:times}$T_{\rm build}$ and $R_{\rm build}$ report the time in seconds and memory in MB, respectively, required
for building the interior ItI operator and constructing the discretized integral equation (\ref{eq:int2}).
$T_{\rm solve}$ reports the time in seconds required to invert the discretized system, while $R_{\rm solve}$
reports the memory in MB to store the inverse.  $T_{\rm apply}$ reports the time in seconds required to apply the
inverse to the incident wave dependent data. This table is independent of the choice of
potential or wavenumber.}
\end{table}

\begin{figure}[ht] 
\begin{tabular}{ccc}
\setlength{\unitlength}{1mm}
\begin{picture}(70,70)
\put(-15,0){\includegraphics[height=70mm]{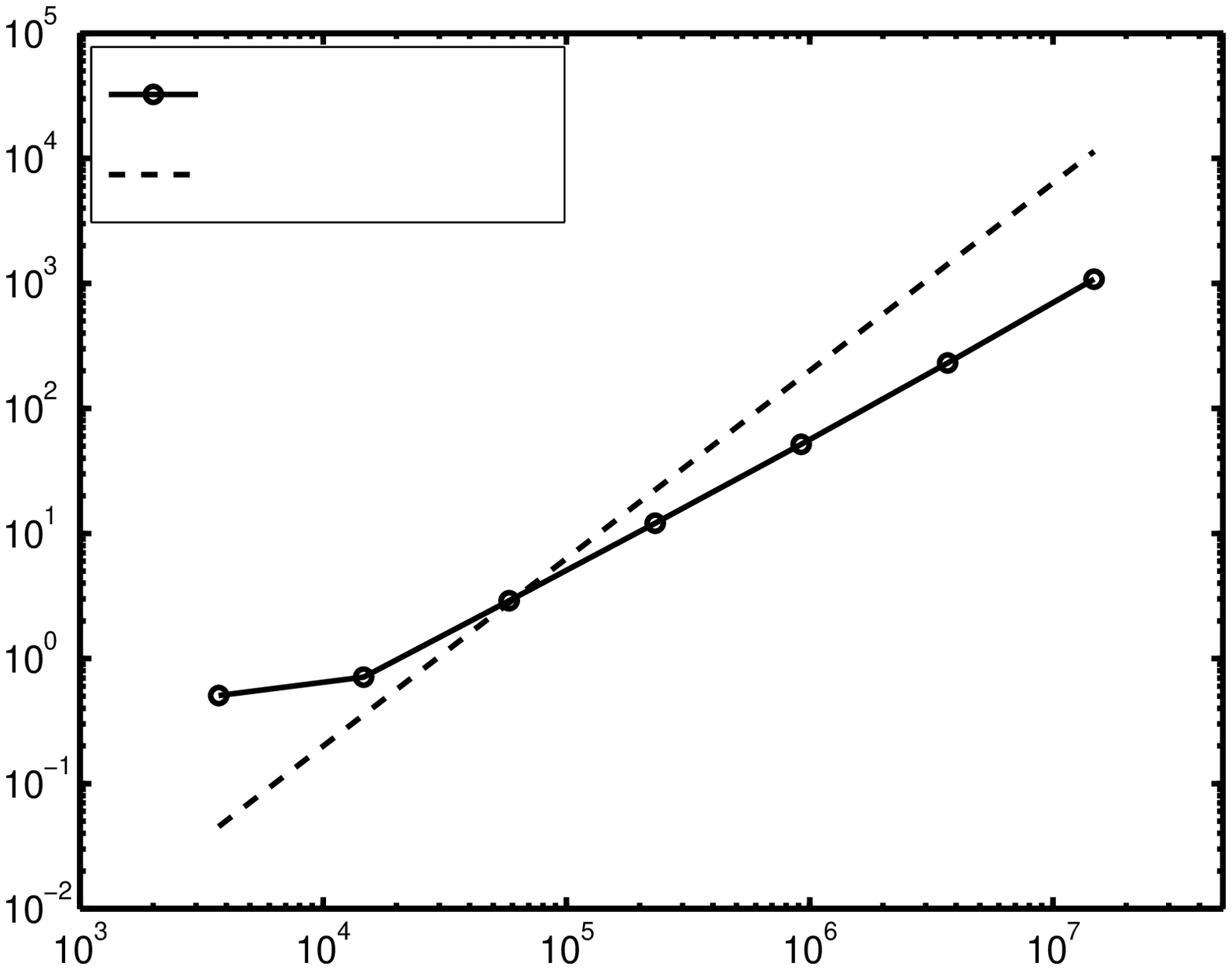}}
\put(-13,25){\rotatebox{90}{Time in seconds}}
\put(33,00){$N$}
\put(7,60){$T_{\rm build}$}
\put(7,54){$C_1\,N^{3/2}$}
\end{picture}
& \mbox{}\hspace{20mm}\mbox{} &
\setlength{\unitlength}{1mm}
\begin{picture}(70,70)
\put(-15,0){\includegraphics[height=70mm]{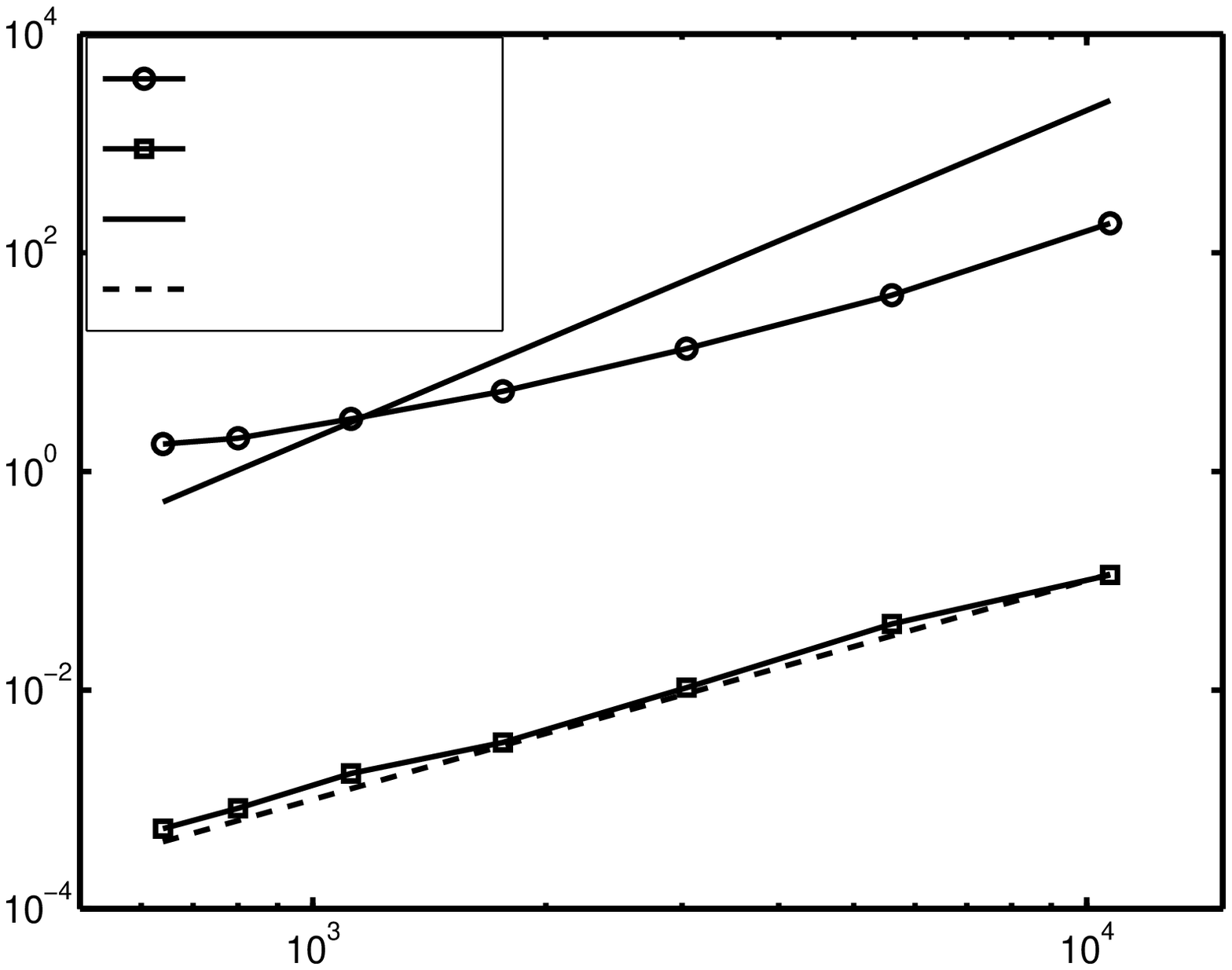}}
\put(-13,25){\rotatebox{90}{Time in seconds}}
\put(33,00){$n$}
\put(7,61){$T_{\rm solve}$}
\put(7,56){$T_{\rm apply}$}
\put(7,51){$C_2\, n^{3}$}
\put(7,46){$C_2\,n^{2}$}
\end{picture}\\
(a) && (b)
\end{tabular}
\caption{Time in seconds for each step in the proposed method. (For the comparison power law graphs,
the prefactors are $C_1 = 2\times 10^{-7}$ and $C_2 = 2\times 10^{-9}$.)}
\label{fig:times}
\end{figure}

\section{Concluding remarks}

This paper presents a robust, high accuracy direct method for solving
scattering problems involving smoothly varying media.
Numerical results show that the method converges to high order,
as expected, for choices of refractive index that are representative
of challenging problems that occur in applications.  Namely, for
problems dominated either by propagation (lenses) or by resonances
(a bandgap photonic crystal),
with of order 100 wavelengths on a side, the method converges to around 9-digit
accuracy with $3.7$ million unknowns.
The method is ideal for
problems where the far field scattering is desired for multiple incident
waves, since each additional incident wave requires merely applying a
dense matrix to its boundary data.
For example, a problem involving $14$ million unknowns
requires $21$ minutes of precomputation (to build the necessary operators),
but each additional solve takes approximately 0.1 seconds.
As discussed in section~\ref{sec:N}, for
low frequency problems,
these timings, and asymptotic behavior, could be improved
by replacing the dense linear algebra by faster algorithms exploiting
compressed representations.
One remaining open question is the existence of a convenient second-kind formulation which
involves the ItI map (and not the DtN map) of the domain $\Omega$.

\section*{Acknowledgments}

We are grateful for a helpful discussion with Michael Weinstein.
The work of AHB is supported by NSF grant DMS-1216656;
the work of PGM is supported by NSF grants DMS-0748488 and DMS-0941476.

\appendix

\section{Reference solution for plane wave scattering from radial potentials}
\label{s:ref}

In this appendix we describe how we generate reference solutions
with around 13 digits of accuracy
for the scattering problem from smooth radially-symmetric
potentials
such as
\be
b(r) = \pm 1.5 e^{-160r^2},
\label{br}
\ee
which are needed in section~\ref{sec:acc}.
Here $(r,\theta)$ are polar coordinates;
in what follows $(x_1,x_2)\in\mathbb{R}^2$ indicate Cartesian coordinates.
We choose a solution domain radius
$R>0$ such that $b$ is numerically negligible outside the ball $r<R$.
A plane wave incident in the positive $x_1$-direction
is decomposed into a polar Fourier (``angular momentum'') basis via the Jacobi--Anger expansion \cite[10.12.5]{dlmf},
$$
e^{i\kappa x_1} = J_0(\kappa r) + 2 \sum_{l=1}^\infty i^l J_l(\kappa r)\cos l\theta
~.
$$
We write $J_l(z) = (H^{(1)}_l(z) + H^{(2)}_l(z))/2$,
and then notice that the effect of the potential $b$ on this field is
to modify only the {\em outgoing} scattering coefficients.
Thus, restricting to a maximum order $L$, the full field becomes
\begin{equation}
u(r,\theta) \approx \half[H^{(1)}_0(\kappa r) + a_0 H^{(2)}_0(\kappa r)]
+ \sum_{l=1}^L i^l [H^{(1)}_l(\kappa r) + a_l H^{(2)}_l(\kappa r)] \cos l\theta,
\qquad
r>R~.
\label{hexp}
\end{equation}
The coefficients $\{a_l\}$ are known as {\em scattering phases};
by flux conservation they lie on the unit circle if $b(r)$ is a real-valued
function.
Convergence with respect to $L$ is exponential,
once $L$ exceeds $\kappa R$.
For the case of \eqref{br} we choose $R=0.5$ and $L=30$.

The phases are found in the following way. 
For each $l=0,\dots L$ we solve the homogeneous radial ODE,
$$
u_l'' + \frac{1}{r} u_l' + \left[ -\frac{l^2}{r^2} + \left( 1 - b(r)\right)\kappa^2\right] u_l = 0,
\qquad 0<r<R
$$
with initial conditions that correspond to a regular solution
of the form $u_l(r) \sim c r^l$ as $r\to 0^+$
(we implement the initial condition
by restricting the domain to $[r_0,R]$ for some small number
$r_0>0$ chosen such that the solution growing with increasing $r$
dominates sufficiently over the decaying one).
For the numerical solution we use
MATLAB's {\tt ode45} with machine precision requested for absolute
and relative tolerances.
(We note that the standard transformation $u(r) = r^lU(r)$ which mollifies the behavior at $r=0$ resulted in no improvement in accuracy.)
After extracting each interior
solution's Robin constant $\beta_l := u'_l(R)/u_l(R)$,
and matching value and derivative to \eqref{hexp} at $r=R$,
we get after simplification,
$$a_l = -\frac{\alpha_l^\ast}{\alpha_l},
\qquad \mbox{where } \;
\alpha_l = \kappa {H^{(1)}_l}'(\kappa R) - H^{(1)}_l(\kappa R) \beta_l
~,
$$
which completes the recipe for the phases.
The computation time required is a few seconds, due to the large number
of steps required by {\tt ode45}.
A simple accuracy test is independence of the phases with respect to variation
in $R$.
Values of $u(r,\theta)$ for $r\ge R$ may then be found via evaluating the sum
in \eqref{hexp}, and for $r<R$ by summation of the interior
solutions $\{u_l(r)\}$.



\end{document}

%% file: macros.tex
\newcommand{\pvct}[1]{\bm{#1}}

\newcommand{\vct}[1]{\bm{\mathsf{#1}}}

\newcommand{\uu}{\vct{u}}

\newcommand{\www}{\vct{w}}

\newcommand{\sss}{\vct{f}}        

\newcommand{\ww}{\vct{g}}          

\newcommand{\mtx}[1]{\bm{\mathsf{#1}}}

\newcommand{\mtwo}[4]{\left[\begin{array}{cc} #1 & #2 \\ #3 & #4 \end{array}\right]}

\newcommand{\vtwo}[2]{\left[\begin{array}{cc} #1 \\ #2 \end{array}\right]}

\newcommand{\pgnotate}[1]{}

\newcommand{\lsp}{\vspace{3mm}}

\newcommand{\ngauss}{q}

\newcommand{\ncheb}{p}

\newcommand{\bi}{\begin{itemize}}

\newcommand{\ei}{\end{itemize}}

\newcommand{\ben}{\begin{enumerate}}

\newcommand{\een}{\end{enumerate}}

\newcommand{\be}{\begin{equation}}

\newcommand{\ee}{\end{equation}}

\newcommand{\bea}{\begin{eqnarray}}

\newcommand{\eea}{\end{eqnarray}}

\newcommand{\ba}{\begin{align}}

\newcommand{\ea}{\end{align}}

\newcommand{\bse}{\begin{subequations}}

\newcommand{\ese}{\end{subequations}}

\newcommand{\bc}{\begin{center}}

\newcommand{\ec}{\end{center}}

\newcommand{\bfi}{\begin{figure}}

\newcommand{\efi}{\end{figure}}

\newcommand{\bmp}[1]{\begin{minipage}{#1}}

\newcommand{\emp}{\end{minipage}}

\newcommand{\bp}{\begin{proof}}

\newcommand{\ep}{\end{proof}}

\newcommand{\tbox}[1]{{\mbox{\tiny #1}}}

\newcommand{\mbf}[1]{{\mathbf #1}}

\newcommand{\half}{\mbox{\small $\frac{1}{2}$}}

\newcommand{\C}{\mathbb{C}}

\newcommand{\R}{\mathbb{R}}

\newcommand{\RR}{\mathbb{R}^2}

\DeclareMathOperator{\re}{Re}

\newcommand{\pO}{{\partial\Omega}}

\newcommand{\Ng}{\ngauss}

\newcommand{\Nc}{\ncheb}

\newcommand{\om}{\omega} 

\newcommand{\xx}{\pvct{x}}   
\newcommand{\yy}{\pvct{y}}   

\newcommand{\ui}{u^\tbox{i}}   

\newcommand{\us}{u^\tbox{s}}   

\newcommand{\dtn}{T}

\newcommand{\iti}{R}

\newcommand{\ub}{u|_\pO}   

\newcommand{\Dx}{{\mtx{D}^{(1)}}}  

\newcommand{\Dy}{{\mtx{D}^{(2)}}}

\newcommand{\NN}{{\mtx{N}}}

\newcommand{\itim}{\mtx{\iti}}

\newcommand{\dtnm}{\mtx{\dtn}}

%% file: main_archiv.bbl
\begin{thebibliography}{}

\bibitem{Ang}
W.~Ang.
\newblock {\em A beginner's course in boundary element methods}.
\newblock Universal Publishers, USA, 2007.

\bibitem{pollution}
I.~M. Babuska and S.~A. Sauter.
\newblock Is the pollution effect of the {FEM} avoidable for the {H}elmholtz
  equation considering high wave numbers?
\newblock {\em SIAM J. Numer. Anal.}, 34(6):2392--2423, 1997.

\bibitem{underwater}
A.~Bayliss, C.~I. Goldstein, and E.~Turkel.
\newblock The numerical solution of the {H}elmholtz equation for wave
  propagation problems in underwater acoustics.
\newblock {\em Comput. Math. Appl.}, 11(7--8):655–--665, 1985.

\bibitem{bayliss85}
A.~Bayliss, C.~I. Goldstein, and E.~Turkel.
\newblock On accuracy conditions for the numerical computation of waves.
\newblock {\em J. Comput. Phys.}, 59(3):396--404, 1985.

\bibitem{2008_bebendorf_book}
M.~Bebendorf.
\newblock {\em Hierarchical matrices}, volume~63 of {\em Lecture Notes in
  Computational Science and Engineering}.
\newblock Springer-Verlag, Berlin, 2008.

\bibitem{2010_borm_book}
S.~B{\"o}rm.
\newblock {\em Efficient numerical methods for non-local operators}, volume~14
  of {\em EMS Tracts in Mathematics}.
\newblock European Mathematical Society (EMS), Z\"urich, 2010.

\bibitem{britt10}
S.~Britt, S.~Tsynkov, and E.~Turkel.
\newblock Numerical simulation of time-harmonic waves in inhomogeneous media
  using compact high order schemes.
\newblock {\em Commun. Comput. Phys.}, 9(3):520--541, 2011.

\bibitem{2002_chen_direct_lippman_schwinger}
Y.~Chen.
\newblock A fast, direct algorithm for the {L}ippmann-{S}chwinger integral
  equation in two dimensions.
\newblock {\em Adv. Comput. Math.}, 16(2-3):175--190, 2002.

\bibitem{2013_yu_chen_totalwave}
Y.~Chen.
\newblock Total wave based fast direct solver for volume scattering problems.
\newblock {\em arxiv.org.}, \#1302.2101, 2013.

\bibitem{coltonkress}
D.~Colton and R.~Kress.
\newblock {\em Inverse acoustic and electromagnetic scattering theory},
  volume~93 of {\em Applied Mathematical Sciences}.
\newblock Springer-Verlag, Berlin, second edition, 1998.

\bibitem{2008_duan_rokhlin}
R.~Duan and V.~Rokhlin.
\newblock High-order quadratures for the solution of scattering problems in two
  dimensions.
\newblock {\em J. Comput. Physics}, 228(6):2152--2174, 2009.

\bibitem{ABC_1977}
B.~Engquist and A.~Majda.
\newblock Absorbing boundary conditions for the numerical simulation of waves.
\newblock {\em Math. Comp.}, 31:629--651, 1977.

\bibitem{2011_engquist_ying_PML}
B.~Engquist and L.~Ying.
\newblock Sweeping preconditioner for the {H}elmholtz equation: Moving
  perfectly matched layers.
\newblock {\em Multiscale Mod. Sim.}, 9(2):686--710, 2011.

\bibitem{3dmicrowave}
Q.~Fang, P.~M. Meaney, and K.~D. Paulsen.
\newblock Viable three-dimensional medical microwave tomography: Theory and
  numerical experiments.
\newblock {\em IEEE Trans. Antennas Propag.}, 58(2):449–--458, 2010.

\bibitem{friedlander}
L.~Friedlander.
\newblock Some inequalities between {D}irichlet and {N}eumann eigenvalues.
\newblock {\em Arch. Rational Mech. Anal.}, 116(2):153--160, 1991.

\bibitem{ONspectralcomposite}
A.~Gillman and P.~Martinsson.
\newblock A direct solver with {$O(N)$} complexity for variable coefficient
  elliptic {PDE}s discretized via a high-order composite spectral collocation
  method, 2013.
\newblock In review.

\bibitem{2012_martinsson_FDS_survey}
A.~Gillman, P.~Young, and P.~Martinsson.
\newblock A direct solver {O(N)} complexity for integral equations on
  one-dimensional domains.
\newblock {\em Front. Math. China}, 7:217--247, 2012.
\newblock 10.1007/s11464-012-0188-3.

\bibitem{hackbusch}
W.~Hackbusch.
\newblock A sparse matrix arithmetic based on {H}-matrices; {P}art {I}:
  {I}ntroduction to {H}-matrices.
\newblock {\em Computing}, 62:89--108, 1999.

\bibitem{gen_quad}
S.~Hao, A.~H. Barnett, P.~G. Martinsson, and P.~Young.
\newblock High-order accurate {Nystr\"om} discretization of integral equations
  with weakly singular kernels on smooth curves in the plane, 2013.
\newblock in press, {\it Adv.\ Comput.\ Math.}

\bibitem{FDS_helmholtz_PML}
E.~Heikkola, T.~Rossi, and J.~Toivanen.
\newblock Fast direct solution of the {H}elmholtz equation with a perfectly
  matched layer/an absorbing boundary condition.
\newblock {\em Int. J. Numer. Meth. Engng}, 57:2007--2025, 2003.

\bibitem{2012_ho_greengard_fastdirect}
K.~Ho and L.~Greengard.
\newblock A fast direct solver for structured linear systems by recursive
  skeletonization.
\newblock {\em SIAM J. Sci. Comput.}, 34(5):A2507--A2532, 2012.

\bibitem{kirschmonk94}
A.~Kirsch and P.~Monk.
\newblock An analysis of the coupling of finite-element and {Nystr\"om} methods
  in acoustic scattering.
\newblock {\em IMA J. Numer. Anal.}, 14:523--544, 1994.

\bibitem{1998_Kopriva}
D.~Kopriva.
\newblock A staggered-grid multidomain spectral method for the compressible
  {Navier-Stokes} equations.
\newblock {\em J. Comput. Phys}, 143(1):125--158, 1998.

\bibitem{2009_martinsson_FEM}
P.~Martinsson.
\newblock A fast direct solver for a class of elliptic partial differential
  equations.
\newblock {\em J. Sci. Comput.}, 38(3):316--330, 2009.

\bibitem{2012_martinsson_spectralcomposite}
P.~Martinsson.
\newblock A direct solver for variable coefficient elliptic {PDE}s discretized
  via a composite spectral collocation method.
\newblock {\em J. Comput. Phys.}, 242:460--479, 2013.

\bibitem{mclean2000}
W.~McLean.
\newblock {\em Strongly elliptic systems and boundary integral equations}.
\newblock Cambridge University Press, Cambridge, 2000.

\bibitem{nichollsnigam04}
D.~P. Nicholls and N.~Nigam.
\newblock Exact non-reflecting boundary conditions on general domains.
\newblock {\em J. Computat. Phys.}, 194(1):278--303, 2004.

\bibitem{dlmf}
F.~Olver, D.~Lozier, R.~Boisvert, and C.~Clark, editors.
\newblock {\em {NIST} Handbook of Mathematical Functions}.
\newblock Cambridge University Press, 2010.
\newblock {\tt http://dlmf.nist.gov}.

\bibitem{2003_pfeiffer}
H.~Pfeiffer, L.~Kidder, M.~Scheel, and S.~Teukolsky.
\newblock A multidomain spectral method for solving elliptic equations.
\newblock {\em Comput. Phys. Commun.}, 152(3):253--273, 2003.

\bibitem{salzer}
H.~E. Salzer.
\newblock {L}agrangian interpolation at the {C}hebyshev points {$x_{n,\nu}
  \equiv \cos(\nu \pi /n)$}, {$\nu=0(1)n$}; some unnoted advantages.
\newblock {\em Comput. J.}, 15:156--159, 1972.

\bibitem{2011_ying_nested_dissection_3D}
P.~Schmitz and L.~Ying.
\newblock A fast direct solver for elliptic problems on {C}artesian meshes in
  {3D}, 2011.
\newblock preprint.

\bibitem{2011_ying_nested_dissection_2D}
P.~Schmitz and L.~Ying.
\newblock A fast direct solver for elliptic problems on general meshes in {2D}.
\newblock {\em J. Comput. Phys.}, 231(4):1314 -- 1338, 2012.

\bibitem{tref}
L.~Trefethen.
\newblock {\em Spectral Methods in {MATLAB}}.
\newblock SIAM, Philadelphia, 2000.

\bibitem{tyrt}
E.~E. Tyrtyshnikov.
\newblock {\em A brief introduction to numerical analysis}.
\newblock Birkh\"auser Boston, 1997.

\bibitem{wadbro10}
E.~Wadbro and M.~Berggren.
\newblock High contrast microwave tomography using topology optimization
  techniques.
\newblock {\em J. Comput. Appl. Math.}, 234:1773--1780, 2010.

\bibitem{seismic3D}
S.~Wang, M.~V. de~Hoop, and J.~Xia.
\newblock On {3D} modeling of seismic wave propagation via a structured
  parallel multifrontal direct {H}elmholtz solver.
\newblock {\em Geophysical Prospecting}, 59(5):857--873, 2011.

\bibitem{2009_xia_superfast}
J.~Xia, S.~Chandrasekaran, M.~Gu, and X.~Li.
\newblock Superfast multifrontal method for large structured linear systems of
  equations.
\newblock {\em SIAM J. Matrix Anal. Appl.}, 31(3):1382--1411, 2009.

\bibitem{2010_gu_xia_HSS}
J.~Xia, S.~Chandrasekaran, M.~Gu, and X.~S. Li.
\newblock Fast algorithms for hierarchically semiseparable matrices.
\newblock {\em Numer. Linear Algebra Appl.}, 17(6):953--976, 2010.

\bibitem{2000_Yang}
B.~Yang and J.~Hesthaven.
\newblock Multidomain pseudospectral computation of {M}axwell's equations in
  {3-D} general curvilinear coordinates.
\newblock {\em Appl. Numer. Math.}, 33(1--4):281 -- 289, 2000.

\end{thebibliography}
